\documentclass[12pt]{amsart}

\usepackage[alphabetic]{amsrefs}
\usepackage[T1]{fontenc}
\usepackage{txfonts}
\usepackage{times}
\usepackage{amsmath, amsthm, amssymb, amscd}
\usepackage{verbatim}
\usepackage{tikz}
\usepackage{cite}
\usepackage{color}
\usepackage{a4wide}
\usepackage[colorlinks=true,linkcolor=darkgray,citecolor=darkgray]{hyperref}
\usepackage{graphicx}
\usepackage{wrapfig}
\usepackage{caption}
\usepackage{subcaption}

\newtheorem{theorem}{Theorem}[section]

\newtheorem{proposition}[theorem]{Proposition}
\newtheorem{corollary}[theorem]{Corollary}
\newtheorem{lemma}[theorem]{Lemma}

\newtheorem{question}[theorem]{Question}

\theoremstyle{definition}

\newtheorem{remark}[theorem]{Remark}

\newtheorem{definition}[theorem]{Definition}

\newenvironment{psmallmatrix}
  {\left(\begin{smallmatrix}}
  {\end{smallmatrix}\right)}

\newcommand{\MCG}{\mathrm{Mod}}
\newcommand{\PMCG}{\mathrm{PMod}}
\newcommand{\Aff}{\mathrm{Aff}}
\newcommand{\Sep}{\mathrm{Sep}}
\newcommand{\Sing}{\mathrm{Sing}}
\newcommand{\Fix}{\mathrm{Fix}}
\newcommand{\PSL}{\mathrm{PSL}}
\newcommand{\SL}{\mathrm{SL}}
\newcommand{\SO}{\mathrm{SO}}
\newcommand{\Ends}{\mathrm{Ends}}
\def\cG{\mathcal{G}}
\def\cT{\mathcal{T}}
\def\cH{\mathcal{H}}
\def\cV{\mathcal{V}}
\def\R{\mathbb{R}}
\def\C{\mathbb{C}}
\def\N{\mathbb{N}}
\def\Z{\mathbb{Z}}

\begin{document}
\title[Loxodromic elements in big mapping class groups]{Loxodromic elements in big mapping class groups via the Hooper-Thurston-Veech construction}


\author{Israel Morales and Ferr\'an Valdez}
\date{\today}

\address{
Ferr\'an Valdez \newline
Centro de Ciencias Matem\'aticas, UNAM, Campus Morelia, C.P. 58190, Morelia, \newline
Michoac\'an,  M\'exico.
}
\email{ferran@matmor.unam.mx}
\address{
Israel Morales \newline
Centro de Ciencias Matem\'aticas, UNAM, Campus Morelia, C.P. 58190, Morelia, \newline
Michoac\'an,  M\'exico.
}
\email{fast.imj@gmail.com}

\begin{abstract}
Let $S$ be an infinite-type surface and $p\in S$. We show that the Thurston-Veech construction for pseudo-Anosov elements, adapted for infinite-type surfaces, produces infinitely many loxodromic elements for the action of $\MCG(S;p)$ on the loop graph $L(S;p)$ that do not leave any finite-type subsurface $S'\subset S$ invariant. Moreover, in the language of~\cite{BaWa18B}, Thurston-Veech's construction produces loxodromic elements of any weight. As a consequence of Bavard and Walker's work, any subgroup of $\MCG(S;p)$ containing two "Thurston-Veech loxodromics" of different weight has an infinite-dimensional space of non-trivial quasimorphisms.
\end{abstract}

\maketitle

\section{Introduction}

Let $S$ be an orientable infinite-type surface, $p\in S$ a marked point\footnote{Through this text we think of $p$ as either a marked point or a marked puncture.} and $\MCG(S;p)$
the quotient of ${\rm Homeo}^+(S;p)$ by isotopies which fix $p$ for all times. This group is related to the (big) mapping class group\footnote{$\MCG(S)$ denotes ${\rm Homeo}^+(S;p)$ modulo isotopy.} of $S$ via Birman's exact sequence:
$$
1\longrightarrow\pi_1(S,p)\overset{Push}{\longrightarrow}\MCG(S;p)\overset{Forget}{\longrightarrow}\MCG(S)\longrightarrow 1.
$$
The group $\MCG(S;p)$ acts by isometries on the (Gromov-hyperbolic, infinite-diameter) loop graph $L(S;p)$, see~\cite{BaWa18A} and \cite{BaWa18B}. Up to date the only known examples of loxodromic elements for this action are:
\begin{enumerate}
\item The loxodromic element $h\in S=\mathbb{S}^2\setminus\{p=\infty \cup \rm Cantor\}$ defined by J. Bavard in~\cite{Bavard16}. In this case $\MCG(S;p)=\MCG(S)$, $h$ is not in the pure mapping class group $\PMCG(S)$, does not preserve any finite type subsurface and, in the language of~\cite{BaWa18B}, has weight 1.
\item Elements defined by pseudo-Anosov homeomorphisms supported in a finite-type subsurface $S'\subset S$ containing $p$. All these live in $\PMCG(S,p)$. Moreover, for any given $m\in\N$, $S'$ can be chosen so that the loxodromic in question has weight $m$.
\end{enumerate}
In~\cite{BaWa18B}, the authors remark that \emph{it would be interesting to construct examples of loxodromic elements of weight greater than 1 which do not preserve any finite type subsurface (up to isotopy)}. 

The purpose of this article is to show that such examples can be obtained by adapting the Thurston-Veech construction for pseudo-Anosov elements (see ~\cite{Thurston88}, ~\cite{Veech89} or~\cite{FarbMArgalit12}) to the context of infinite-type surfaces. This adaptation is an extension of Thurston and Veech's ideas built upon previous work by Hooper~\cite{Hooper15}, hence we call it the Hooper-Thurston-Veech construction.  Roughly speaking, we show that if one takes as input an appropiate pair of multicurves $\alpha$, $\beta$ whose union fills $S$, then the subgroup of $\MCG(S,p)$ generated by the (right) multitwists $T_\alpha$ and $T_\beta$ contains infinitely many loxodromic elements. More precisely, our main result is:


\begin{theorem}
	\label{THM:Loxodromics}
Let $S$ be an orientable infinite-type surface, $p\in S$ a marked point and $m\in\mathbb{N}$. Let $\alpha=\{\alpha_i\}_{i\in I}$ and $\beta=\{\beta_j\}_{j\in J}$ be multicurves in minimal position whose union fills $S$ and such that:
\begin{enumerate}
\item the configuration graph\footnote{The vertices of this graph are $I\cup J$. There is an edge between $i\in I$ and $j\in J$ for every point of intersection between $\alpha_i$ and $\beta_j$.} $\mathcal{G}(\alpha\cup\beta)$ is of finite valence\footnote{Valence here means the supremum of $degree(v)$ where $v$ is a vertex of the graph in question.},
\item for some fixed $N\in\mathbb{N}$, every connected component $D$ of $S\setminus\alpha\cup\beta$ is a polygon or a once-punctured polygon\footnote{The boundary of any connected component of $S\setminus\alpha\cup\beta$ is formed by arcs contained in the curves forming $\alpha\cup\beta$ and hence we can think of them as (topological) polygons.} with at most $N$ sides and
\item the connected component of $S\setminus\alpha\cup\beta$ containing $p$  is a $2m$-polygon.
\end{enumerate}
If  $T_\alpha$, $T_\beta\in\MCG(S;p)$ are the (right) multitwists w.r.t. $\alpha$ and $\beta$ respectively then any $f\in\MCG(S;p)$ in the positive semigroup generated by $T_\alpha$ and $T_\beta^{-1}$ given by a word on which both generators appear is a loxodromic element of weight $m$ for the action of $\MCG(S;p)$ on the loop graph $L(S;p)$.
\end{theorem}

\begin{remark}
    During the 2019 AIM-Workshop on Surfaces of Infinite Type we learned that Abbott, Miller and Patel have also a construction of loxodromic mapping classes whose support is an infinite type surface~\cite{AbMiPa}. Their examples are obtained via a composition of handle shifts and live in the complement of the closure of the subgroup of $\MCG(S,p)$ defined by homeomorphisms with compact support. In contrast, loxodromic elements given by Theorem~\ref{THM:Loxodromics} are limits of compactly supported mapping classes.
\end{remark}

As we explain in Section~\ref{Subsection:LoopGraph}, the weight of a loxodromic element $f\in\MCG(S;p)$ is defined by Bavard and Walker using a precise description of the (Gromov) boundary $\partial L(S;p)$ of the loop graph. For finite-type surfaces, if $f$ is a pseudo-Anosov having a singularity at $p$, this number corresponds to the number of separatrices based at $p$ of an $f$-invariant transverse measured foliation. This quantity is relevant because, as shown in~\cite{BaWa18B} and using the language of Bestvina and Fujiwara~\cite{BestvinaFujiwara02}, loxodromic elements with different weights are independent and anti-aligned. This has several consequences, for example the work of Bavard-Walker~\cite{BaWa18A} gives for free the following:

\begin{corollary}
Let $f,g$ be two loxodromic elements in $\MCG(S;p)$ as in Theorem~\ref{THM:Loxodromics} and suppose that their weights are $m\neq m'$. Then any subgroup of $\MCG(S;p)$ containing them has an infinite-dimensional space of non-trivial quasimorphisms.
\end{corollary}

Applications for random-walks on the loop graph with respect to probability distributions supported on countable subgroups of $\MCG(S;p)$ can be easily deduced from recent work of Maher-Tiozzo~\cite{MaherTiozzo18}.

On the other hand, recent work by Rasmussen~\cite{Rass19} implies that the mapping classes given by Theorem~\ref{THM:Loxodromics} are not WWPD in the language of Bestvina-Bromberg-Fujiwara~\cite{BeBroFu15}.

\emph{About the proof of Theorem~\ref{THM:Loxodromics}}. As in the case of Thurston's work, our proof relies on the existence of a flat surface structure $M$ on $S$, having a conical singularity at $p$, for which the Dehn-twists $T_\alpha$ and $T_\beta$ are affine automorphisms. In the case of finite-type surfaces, this structure is unique (up to scaling) and its existence is guaranteed by the Perron-Frobenius theorem. For infinite-type surfaces the presence of such a flat surface structure is guaranteed once one can find a positive eigenfunction of the adjacency operator on the (infinite bipartite) configuration graph $\mathcal{G}(\alpha\cup\beta)$. Luckly, the spectral theory of infinite graphs in this context secures the existence of uncountably many flat surface structures (which are not related by scaling) on which the Dehn-twists $T_\alpha$ and $T_\beta$ are affine automorphisms. The main difficulty we encounter is that the description of the Gromov boundary $\partial L(S;p)$ needed to certify that $f$ is a loxodromic element depends on a hyperbolic structure on $S$ which, a priori, is not quasi-isometric to any of the aforementioned flat surface structures. To overcome this difficulty we propose arguments which are mostly of topological nature.

We strongly believe that Theorem~\ref{THM:Loxodromics} does not describe all possible loxodromics living in the group generated by $T_\alpha$ and $T_\beta$.

\begin{question}
Let $\alpha$ and $\beta$ be as in Theorem~\ref{THM:Loxodromics}. Is every element $f$ in the group generated by $T_\alpha$ and $T_\beta$ given by a word on which both generators appear loxodromic? In particular, is $T_\alpha T_\beta$ loxodromic?
\end{question}

 We spend a considerable part of this text in the proof of the next result, which guarantees that Theorem~\ref{THM:Loxodromics} is not vacuous.
\begin{theorem}
	\label{THM:Multicurves}
Let $S$ be an infinite-type surface, $p\in S$ a marked point and $m\in\mathbb{N}$. Then there exist two multicurves $\alpha$ and $\beta$ whose union fills $S$ and such that:
\begin{enumerate}
\item the configuration graph $\mathcal{G}(\alpha\cup\beta)$ is of finite valence,
\item every connected component $D_i$ of $S\setminus\alpha\cup\beta$  which does not contain the pont $p$ is a polygon or a once-punctured polygon with $n_i$ sides, where $n_i\leq N:=\max\{8,m\}$, and
\item $p$ is  contained in a connected component of $D_j$ of $S\setminus\alpha\cup\beta$ whose boundary $\partial D_j$ is a $2m$-polygon.
\end{enumerate}
\end{theorem}

A crucial part on the proof of this result is to find, for \emph{any} infinite-type surface $S$, a simple way to portrait $S$. We call this a topological normal form. Once this is achieved, we give a recipe to construct the curves $\alpha$ and $\beta$ explicitly.

On the other hand, we find phenomena proper to big mapping class groups:

\begin{corollary}
    \label{Corollary:LoxodromicsConverging2Elliptic}
Let $S$ be the Loch-Ness monster\footnote{$S$ has infinite genus and only one end.} and consider the action of $\MCG(S;p)$ on the loop graph. Then there exist
a sequence of loxodromic elements $(f_n)$ in $\MCG(S;p)$ which converge in the compact-open topology to a non-trivial elliptic element.
\end{corollary}

\begin{theorem}
    \label{COR:FlatAmbivalentHomeos}
There exits a family of translation surface structures $\{M_\lambda\}_{\lambda\in[2,+\infty]}$ on a  Loch Ness monster $S$ and $f\in\MCG(S)$ such that:
\begin{itemize}
\item For every $k\in\N$, $f^k$ does not fix any isotopy class of essential simple closed curve in $S$.
\item If $\tau_\lambda:\Aff(M_\lambda)\hookrightarrow\MCG(S)$ is the natural map  sending an affine homeomorphism to is mapping class, then $D\tau_\lambda^{-1}(f)\in\PSL(2,\R)$ is parabolic if $\lambda=2$ and hyperbolic for every $\lambda>2$.
\end{itemize}
\end{theorem}

Recall that for finite-type surfaces $S$ a class $f\in\MCG(S)$ such that for every $k\in\N$, $f^k$ does not fix any isotopy class of essential simple closed curve in $S$ is necessarily pseudo-Anosov. In particular, the derivative of any affine representative $\phi\in f$ is hyperbolic.

We want to stress that for many infinite-type surfaces $\MCG(S)$ does not admit an action on a metric space with unbounded orbits. For a more detailed discussion on this fact and the large scale geometry of big mapping class groups we refer the reader to~\cite{DurhamFanoniVlamis18}, ~\cite{MahnRafi20} and references therein.

\emph{Outline}. Section~\ref{Section:Preliminaries} is devoted to preliminaries about the loop graph, its boundary and infinite-type flat surfaces. In Section~\ref{Subsection:HTVConstruction} we present  the Hooper-Thurston-Veech construction\footnote{More precisely, the construction here presented is a particular case of a more general construction built upon work of P. Hooper by the second author and  V. Delecroix. The first version of this more general construction had mistakes that were pointed out by the first author.}. In Section~\ref{Subsection:HTVConstruction} we also proof Theorem~\ref{COR:FlatAmbivalentHomeos}. Finally, Section~\ref{Section:ProofResults} is devoted to the proof of Theorems~\ref{THM:Loxodromics},~\ref{THM:Multicurves} and Corollary~\ref{Corollary:LoxodromicsConverging2Elliptic} (in this order).

\emph{Acknowledgements}. We are greatful to Vincent Delecroix, Alexander Rasmussen and Patrick Hooper for providing crucial remarks  that lead to the proof of Theorem~\ref{THM:Loxodromics}. We also want to thank Vincent Delecroix for letting us include a particular case of the Hooper-Thurston-Veech's construction and Figures~\ref{fig:DualGraph},~\ref{fig:MakingRectanglesEuclidean} and~\ref{Fig:MulticurvesInfiniteStaircase}, which are work in collaboration with the second author (see~\cite{DHV19}). We are greatful to Juliette Bavard and Alden Walker for taking the time to explain the details of their work, and to Chris Leininger for answering all our questions. We want to thank the following institutions and grants for their hospitality and financial support:
Max-Planck Institut f\"{u}r Mathematik, American Institut of Mathematics, PAPIIT IN102018, UNAM, PASDA de la DGAPA, CONACYT and Fondo CONACYT-Ciencia B\'asica's project 283960.






\section{Preliminaries}
    \label{Section:Preliminaries}

\subsection{Infinite-type surfaces} Any orientable (topological) surface $S$ with empty boundary is determiner up to homeomorphism by its genus (possibly infinite) and a pair of nested, totally disconnected, separable, metrizable topological spaces $\Ends_\infty(S)\subset\Ends (S)$ called the space of ends accumulated by genus and the space of ends of $S$, respectively.  Any such pair of nested topological spaces occurs as the space of ends of some orientable surface, see~\cite{Richards63}. On the other hand, $S\cup\Ends(S)$ can be endowed with a natural topology which makes the correspoding space compact. This space is called the Freudenthal end-point compactification of $S$, see~\cite{Raymond60}.

A surface $S$ is of finite (topological) type if its fundamental group is finitely generated. In any other case we say that $S$ is an \emph{infinite-type} surface. All surfaces considered henceforth are orientable.

\subsection{Multicurves} Let $S$ be an infinite-type surface. A collection of essential curves $l$ in $S$ is {{\it locally finite}} if for every $x\in S$ there exist a neighborhood $U$ of $x$ which intersects a finitely many elements of $l$. As surfaces are second-countable topological spaces, any locally finite collection of essential curves is countable.


A \emph{multicurve} in $S$ is a locally finite, pairwise disjoint, and pairwise non-isotopic collection of essential curves in $S$.

Let $\alpha$ be a multicurve in $S$. We say that $\alpha$ \emph{bounds a subsurface} $\Sigma$ of $S$, if the elements of $\alpha$ are exactly all the boundary curves of the closure of $\Sigma$ in $S$. Also, we say that $\Sigma$ is \emph{induced} by $\alpha$ if there exists a subset $\alpha' \subset \alpha$ that bounds $\Sigma$ and there are no elements of $\alpha \smallsetminus \alpha'$ in its interior.

A multicurve $\alpha$ in $S$ is of \emph{finite type} if every connected component of $S \smallsetminus \alpha$ is a finite-type surface.

Finite multicurves (that is, with formed by a finite number of curves) are not necessarily of finite type. On the other hand, there are infinite multicurves which are not of finite type, \emph{e.g.} the blue ("vertical") curves in the right-hand side of Figure \ref{MulticurvasEjemplos}.


Let $\alpha$ and $\beta$ be two multicurves in $S$. We say that $\alpha$ and $\beta$ are in \emph{minimal position} if for every $\alpha_i \in \alpha$ and $\beta_j \in \beta$, $\vert \alpha_i\cap \beta_j\vert$ realizes the minimal number of (geometric) intersection points between a representative in the free isotopy class of $\alpha_i$ and a representative in the homotopy class of $\beta_j$.  For every pair of multicurves one can find representatives in their isotopy classes which are in minimal position.

Let $\alpha$ and $\beta$ be two multicurves in $S$ in minimal position. We say that $\alpha$ and $\beta$ \emph{fill} $S$ if every connected component of $S\smallsetminus \left( \alpha \cup \beta \right)$ is either a disk or a once-punctured disk.

\begin{remark}
	\label{Remark:MulticurvesFiniteType}
Let $\alpha$ and $\beta$ be multicurves. Then:
\begin{enumerate}
\item If $\alpha$ and $\beta$ are of finite type and fill $S$, then every complementary connected component of $\alpha \cup \beta$ in $S$ is a polygon with finitely many sides. The converse is not true, see the left-hand side of Figure \ref{MulticurvasEjemplos}.
\item There are pair of multicurves $\alpha$ and $\beta$ so that $S\smallsetminus \left( \alpha \cup \beta \right)$ has a connected component that is a polygon with infinitely many sides, see the right-hand side of  Figure \ref{MulticurvasEjemplos}.
\end{enumerate}

\begin{figure}[!ht]
\begin{center}
	\includegraphics[width=.80\textwidth]{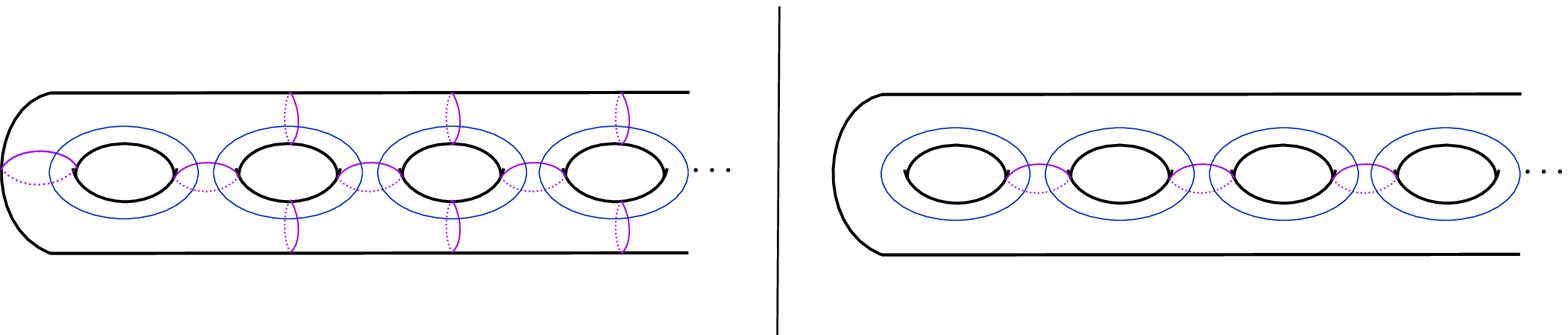}
	\caption{ }
	\label{MulticurvasEjemplos}
\end{center}
\end{figure}
\end{remark}

\subsection{The loop graph and its boundary}
\label{Subsection:LoopGraph}
 In~\cite{BaWa18A} and~\cite{BaWa18B}, Bavard and Walker introduced the loop graph $L(S;p)$ and prove that it is hyperbolic graph on which $\MCG(S,p)$ acts by isometries. Moreover, they made a precise description of the Gromov boundary of $L(S;p)$ in terms of (hyperbolic) geodesics on the Poincar\'e disk. We recall the general aspects of their work in what follows. The exposition follows largely ~\cite{BaWa18A}, \cite{BaWa18B} and~\cite{Rass19}.

\emph{The loop graph}. Let $S$ be an infinite type surface and $p\in S$. In what follows we think of $p$ as a
marked puncture in $S$. The isotopy class of a topological embedding of $\gamma:(0,1)\hookrightarrow S$ is said to be a \emph{loop} if it can be continuously extended to the end-point Freudenthal compactification $S\cup\Ends(S)$ of $S$ with $\gamma(0)=\gamma(1)=p$. On the other hand, if the continuous extension of $\gamma$ satifies that $\gamma(0)=p$ and $\gamma(1)\in\Ends(P)\setminus\{p\}$ we call it a short ray. The loop graph has as vertex set isotopy classes (relative to $p$) of loops and adjacency is defined by disjointness (modulo isotopy) and it is hyperbolic w.r.t to the combinatorial distance, see~\cite{BaWa18B}. In order to describe the Gromov boundary of $L(S;p)$ we need to introduce the completed ray graph.

\emph{Long rays and the completed ray graph}. From now on we fix a hyperbolic metric $\mu$ on $S$ of the first kind\footnote{That is, the Fuchsian group appearing in the uniformization $\mathbb{D}\to S$ has as limit set the whole boundary of the Poincar\'e disk.} for which the marked point $p$ is a cusp. Every short ray or loop has a unique geodesic representative in its isotopy class. We denote by $\pi:\hat{S}\to S$ the infinite cyclic cover of $S$ defined by the (conjugacy class of) cyclic subgroup of $\pi_1(S,\cdot)$ generated by a simple loop around the cusp $p$ and call it \emph{the conical cover of $S$}. The surface $\hat{S}$ is conformally equivalent to a punctured disc and its
unique cusp $\hat{p}$ projects to $p$. We denote by $\partial\hat{S}$ the Gromov boundary $\partial\hat{S}$ from which $\hat{p}$ has been removed. This cover is usefull because for every geodesic representative of a short ray or loop in $S$ there is a unique lift to $\hat{S}$ which is a geodesic with one end in $\hat{p}$ and the other in $\partial\hat{S}$.

A long ray on S is a \emph{simple} bi-infinite geodesic of the form $\pi(\hat{\delta})$, where $\hat{\delta}\subset\hat{S}$ is a geodesic from $\hat{p}$ to $\hat{S}$, which is not a short ray or a loop. By definition, each long ray limits to p at one end and does not limit to any point of $\Ends(S)$ on the other end. The vertices of the completed ray graph $\mathcal{R}(S;p)$ are isotopy classes of loops and short rays, and long rays. Two vertices are adjacent if their geodesic representatives in $(S,\mu)$ are disjoint. As before, we consider the combinatorial metric on $\mathcal{R}(S;p)$ defined by declaring that each edge has length 1.

\begin{theorem}~\cite{BaWa18B}
    \label{THM:CompletedRayGraph}
The completed ray graph $\mathcal{R}(S;p)$ is disconnected. There exist a component containg all loops and short rays, which is of infinite diameter and quasi-isometric to the loop graph. All other connected components are (possibly infinite) cliques and each of them is formed exclusively by long rays.
\end{theorem}

The component of $\mathcal{R}(S;p)$ containg all loops and short rays is called the \emph{main component} of the completed ray graph. Long rays not contained in the main component are called \emph{high-filling} rays and they each one of they form cliques formed exclusively of high-filling rays.

\emph{The Gromov boundary of the loop graph}. Let us denote by $\mathcal{H}(S,p)$ the set of all high-filling rays in $\mathcal{R}(S;p)$. Bavard and Walker endow $\mathcal{H}(S,p)$ with a topology. This topology is based on the notion of two rays $k$-beggining like each other, see Section 4.1 and Definition 5.2.4  in~\cite{BaWa18B}. On the other hand, they define a $\MCG(S;p)$-action on $\mathcal{H}(S;p)$ by homeomorphisms. We sketch this action in what follows. First they show that endpoints of lifts of loops and short rays to the conical cover $\hat{S}$ are dense in $\partial\hat{S}$. Using this, and the fact that mapping classes in $\MCG(S;p)$ permute loops and short rays, they prove that any $\phi\in\MCG(S;p)$ lifts to a homeomorphism of $\hat{S}$ which admits a unique continuous extension to a homeomorphism of $\partial\hat{S}$. Finally, they show that this extension preserves the subset of $\partial \hat{S}$ formed by endpoints of high-filling rays, hence inducing the aforementioned action by homeomorphisms.

\begin{theorem}
Let $\mathcal{E}(S;p)=\mathcal{H}(S;p)/\sim$, where $\sim$ identifies all high-filling rays in the same clique, and endow this set with the quotient topology. Then there exists a $\MCG(S;p)$-equivariant homeomorphism $F:\mathcal{E}(S;p)\to\partial L(S;p)$, where $\partial L(S;p)$ is the Gromov boundary of the loop graph.
\end{theorem}

In consequence any loxodromic element $\phi\in\MCG(S;p)$ fixes two cliques of high-filling rays $\mathcal{C}^-(\phi)$ and $\mathcal{C}^+(\phi)$.
\begin{theorem}[\cite{BaWa18B}, Theorem 7.1.1]
The cliques $\mathcal{C}^-(\phi)$ and $\mathcal{C}^+(\phi)$ are finite and of the same cardinality.
\end{theorem}

This allows to define the \emph{weight} of a loxodromic element $\phi$ as the cardinality of either $\mathcal{C}^-(\phi)$ or $\mathcal{C}^+(\phi)$. As said in the introduction, the importance of the weight of a loxodromic elements is given by the following fact:

\begin{lemma}[\cite{BaWa18B}, Lemma 9.2.7]
\label{Remark:neqWeightEqIndependence}
Let $g,h\in\MCG(S;p)$ be two loxodromic elements with different weights. Then in the language of
Bestvina-Fujiwara~\cite{BestvinaFujiwara02}, $g$ and $h$ are independent and anti-aligned.
\end{lemma}

\subsection{Flat surfaces}
In this section we recall only basic concepts about flat surfaces needed for the rest of the paper. It is important to remark that most of the flat surfaces considered in this text are of infinite type. For a detailed discussion on infinite-type flat surfaces we refer the reader to~\cite{DHV19}.

 We use $x,y$ for the standard coordinates in $\R^2$, $z = x + i y$ the corresponding number in $\C$ and $(r,\theta)$ for polar coordinates $x = r \cos(\theta)$ and $y = r \sin(\theta)$ (or $z = r \exp(i \theta)$). The Euclidean metric $dx^2 + dy^2$ can also be written as $(dr)^2 + (r d\theta)^2$.

Let $S$ be an orientable surface and $g$ is a metric defined on the complement of a discrete set $\Sigma\subset S$. A point $p\in S$ is called a \emph{conical singularity of angle $\pi n$} for some $n\in\mathbb{N}$ if there exists an open neighbourhood $U$ of $p$ such that $(U,g)$ is isometric to $(V,g_n)$, where $V\subset \C^*$ is a (punctured) neighborhood of the origin and $g_n=(dr)^2+(nrd\theta)^2$. If $n=2$ we call $p$ a \emph{regular point}. In general, regular points are not considered as singularities, though as we see in the proof of Theorem~\ref{THM:Loxodromics} sometimes it is convenient to think of them as marked points.

A \emph{flat surface structure} on a topological surface $S$ is the maximal atlas $\mathcal{T} = \{\phi_i:U_i \to \C\}$ where $(U_i)_{i\in\N}$ forms an open covering of $S$, each $\phi_i$ is a homeomorphism from $U_i$ to $\phi(U_i)$ and for each $i,j$ the transition map $\phi_j \circ \phi_i^{-1}: \phi_i(U_i \cap U_j) \to \phi_j(U_i \cap U_j)$ is a translation in $\C$ or a map of the form\footnote{These kind of maps are also called \emph{half-translations} and for this reason flat surfaces containing half-translation are also known as half-translation surfaces.} $z\to -z+\lambda$ for some $\lambda\in\C$.

\begin{definition}
\label{def:FlatSurfaceGeometric}
A \emph{flat surface} is a pair $M=(S,\mathcal{T})$ made of a connected topological surface $S$ and a flat surface structure $\cT$ on $S\setminus\Sigma$, where:
\begin{enumerate}
\item $\Sigma$ is a discrete subset of $S$ and
\item every $z\in\Sigma$ is a conical singularity.
\end{enumerate}
If the transition functions of $\mathcal{T}$ are all translations we call the pair $(S,\mathcal{T})$ a translation surface.
\end{definition}

\begin{remark}
    \label{Remark:OnFlatSurfaces}
In the preceding definition $S$ can be of infinite topological type and $\Sigma$ can be infinite. All points in $M\setminus\Sigma$ are regular. Every flat surface carries a flat metric given by pulling back the Euclidean metric in $\C$. We denote by $\widehat{M}$ the corresponding metric completion and ${\rm Sing}(M)\subset\widehat{M}$ the set of non-regular points, which can be thought as singularities of the flat metric. We stress that the structure of $M$ near a non-regular point is not well understood in full generality, see~\cite{BowmanValdez13} and \cite{Randecker18}.

Every flat surface $M$ which is not a translation surface has a (ramified) double covering $\pi:\widetilde{M}\to M$ such that $\widetilde{M}$ is a translation surface whose atlas is obtained by pulling back via $\pi$ the flat surface structure of $M$. This is called the orientation double covering.
If $z_0\in M$ is a conical singularity of angle $n\pi$ then, if $n$ is even, $\pi^{-1}(z_0)$ is formed by two conical singularities of total angle $n\pi$; whereas if $n$ is odd, $\pi^{-1}(z_0)$ is a conical singularity of total angle $2n\pi$. Hence the branching points of the orientation double covering are the conical singularities in $M$ of angle $n\pi$, with $n$ odd. This will be used in the proof of Theorem~\ref{THM:Loxodromics}.

On the other hand, flat surfaces can be defined using the language of complex geometry in terms of Riemann surfaces and quadratic differentials or by glueing (possibly infinite) families of polygons along their edges. A detailed discussion on these other definitions can be found in the first Chapter of~\cite{DHV19}.

\end{remark}

\noindent\textbf{Affine maps}.
A map $f\in Homeo(M)$ with $f(\Sigma)\subset\Sigma$ is called an \emph{affine automorphisms} if the
restriction $f:M\setminus\Sigma\to M\setminus\Sigma$ to flat charts is an
$\R$-affine map. We denote by $\Aff(M)$ the group of affine homeomorphisms of $M$ and by $\Aff^+(M)$ the subgroup of $\Aff(M)$ made of orientation preserving affine automorphisms (\emph{i.e} their linear part has positive determinant). Remark that the derivative $Df$ of an element $f\in \Aff(M)$ is an element of ${\mathrm{GL}}_2(\R)/\pm Id$.

\noindent\textbf{Translation flows}. For each direction $\theta\in\R/2\pi\Z$ we have a well-defined translation flow $F_{\C,\theta}^t: \C \to \C$ given by $F_{\C,\theta}^t(z) = z + t e^{i \theta}$. This flow defines a constant vector field $X_{\C,\theta}(z):=\frac{\partial F_{\C,\theta}^t}{\partial t}|_{t=0}(z)$.
Now let $M$ be a translation surface and $X_{M,\theta}$ the vector field on $\widehat{M}\setminus{\rm Sing}(M)$ obtained by pulling back $X_{\C,\theta}$ using the charts of the structure. For every $z\in M\setminus\Sing(M)$ let us denote by $\gamma_z:I\to M$, where $I\subset\R$ is an interval containing zero, the maximal integral curve of $X_{M,\theta}$ with initial condition $\gamma_z(0)=z$. We define $F_{M,\theta}^t(z):=\gamma_z(t)$ and call it the \emph{translation flow} on $M$ in direction $\theta$. Let us remark that formally speaking $F^t_{M,\theta}$ is not a flow because the trajectory of the curve $\gamma_z(t)$ may reach a point of $\Sing(M)$ in finite time. A trajectory of the translation flow whose maximal domain of definition is a bounded interval (on both sides) is called a \emph{saddle connection}.
When there is no need to distinguish translation flows in different translation surfaces we abbreviate $F_{M,\theta}^t$ and $X_{M,\theta}$ by $F_{\theta}^t$ and $X_{\theta}$ respectively. If $M$ is a flat surface but not a translation surface the pull back of the vector field $e^{i\theta}$ to $M$ does not define a global vector field on $M$ but it does define a \emph{direction field}. In both cases integral curves defines a foliation $\mathcal{F}_{\theta}$ on $M\setminus \Sing(M)$.

\begin{definition}[Cylinders and strips]
\label{def:CylindersAndStrips}
A \emph{horizontal cylinder} $C_{c,I}$ is a translation surface of the form $([0,c]\times I) / \sim$, where $I\subset\R$ is open (but not necessarily bounded), connected, and where $(0,s)$ is identified with $(c,s)$ for all $s\in I$. The numbers $c$ and $h=|I|$ are called the \emph{circumference} and \emph{height} of the cylinder respectively. The \emph{modulus} of $C_{c,I}$ is the number $\frac{h}{c}$.

 A \emph{horizontal strip} $C_{\infty,I}$ is a translation surface of the form $\R\times I$, where $I$ is a bounded open interval. Analogously, the height of the horizontal strip is $h=|I|$.

An open subset of a translation surface $M$ is called a \emph{cylinder} (respectively a \emph{strip}) in direction $\theta$ (or parallel to $\theta$) if it is isomorphic, as a translation surface, to $e^{-i\theta}C_{c,I}$ (respect. to $e^{-i\theta}C_{\infty,I}$).
\end{definition}

One can think of strips as cylinders of infinite circumference and finite height. For flat surfaces which are not translation surfaces the definition of cylinder still makes sense, though its direction is well defined only up to change of sign.

\begin{definition}
    \label{DefCylinderDecomposition}
Let $M$ be a flat surface and $\theta\in\R/2\pi\Z$ a fixed direction. A collection of maximal cylinders $\{C_i\}_{i\in I}$ parallel to $\theta$ such that $\cup_{i\in I} C_i$ is dense in $M$ is called a \emph{cylinder decomposition} in direction $\theta$.
\end{definition}


\section{The Hooper-Thurston-Veech construction}
	\label{Subsection:HTVConstruction}

The main result of this section is a generalization of the Thurston-Veech construction for infinite-type surfaces. The key ingredient for this generalization is the following:
\begin{lemma}
\label{lem:CylinderDecompImpliesParabolic}
Let $M$ be a flat surface for which there is a cylinder decomposition in the horizontal direction. Suppose that every maximal cylinder in this decomposition has modulus equal to $\frac{1}{\lambda}$ for some $\lambda>0$. Then there exists a unique affine automorphism $\phi_h$ which fixes the boundaries of the cylinders and whose derivative (in $\PSL(2,\R)$) is given by the matrix $\begin{psmallmatrix}1 & \lambda\\ 0 & 1\end{psmallmatrix}$. Moreover, the automorphism $\phi_h$ acts as a Dehn twist along the core curve of each cylinder.
\end{lemma}

In general, if $M$ is a flat surface having a cylinder decomposition in direction $\theta\in\R/2\pi\Z$ for which every cylinder has modulus equal to $\frac{1}{\lambda}$ for some $\lambda\in\R^*$, one can apply to $M$ the rotation $R_\theta\in\SO(2,\R)$ that takes $\theta$ to the horizontal direction and apply the preceding lemma. For example, if $\theta=\frac{\pi}{2}$, then there exists a unique affine automorphism $\varphi_v$ which fixes the boundaries of the vertical cylinders, acting on each cylinder as a Dehn twists and with derivative (in $\PSL(2,\R)$) is given by the matrix $\begin{psmallmatrix}1 & 0\\ -\lambda & 1\end{psmallmatrix}$. In particular, if $M$ is a flat surface having cylinder decompositions in the horizontal and vertical directions such that each cylinder involved have modulus $\frac{1}{\lambda}$, then $\Aff(M)$ has two affine multitwists $\phi_h$ and $\phi_v$ with $D\phi_h=\begin{psmallmatrix}1 & \lambda\\0 & 1\end{psmallmatrix}$ and $D\phi_v=\begin{psmallmatrix}1 & 0\\ -\lambda & 1\end{psmallmatrix}$ in $\PSL(2,\mathbb{R})$.

Let us recall now the Thurston-Veech construction (see [Farb-Margalit], Theorem 14.1):

\begin{theorem}[Thurston-Veech construction]
    \label{THM:ThursonsConstruction}
Let $\alpha=\{\alpha_i\}_{i=1}^n$ and $\beta=\{\beta_j\}_{j=1}^m$ be two multicurves filling a finite type surface $S$. Then there exists $\lambda=\lambda(a,b)\in\mathbb{R}^*$ and a representation $\rho:\langle T_\alpha,T_\beta\rangle\to\PSL(2,\R)$ given by:
$$
T_\alpha\to \begin{pmatrix} 1 & \lambda \\ 0 & 1\end{pmatrix},\hspace{5mm} T_\beta\to =\begin{pmatrix} 1 & 0 \\ -\lambda & 1\end{pmatrix}.
$$
Moreover, an element $f\in\langle T_\alpha,T_\beta\rangle$ is periodic, reducible or pseudo-Anosov according to whether $\rho(f)$ is elliptic, parabolic or hyperbolic.
\end{theorem}

The proof of Theorem~\ref{THM:ThursonsConstruction} uses Lemma~\ref{lem:CylinderDecompImpliesParabolic}. More precisely, one needs to find a flat surface structure on $S$ which admits horizontal and vertical cylinder decompositions $\{H_i\}_{i=1}^n$ and $\{V_j\}_{j= 1}^m$ such that each cylinders has modulus equal to $\frac{1}{\lambda}$ and for which $\alpha_i$ and $\beta_j$ are the core curves of $H_i$ and $V_j$ for each $i=1,\ldots,n$ and $j=1,\ldots,m$, respectively. By the Perron-Frobenius theorem, such a flat structure always exists and is unique up to scaling.


\begin{definition}
\label{def:ConfigurationGraph}
Let $\alpha=\cup_{i\in I}\alpha_i$ and $\beta=\cup_{j\in J}\beta_j$ be two multicurves in a topological surface
$S$ (in minimal position, not necessarily filling).
The \emph{configuration graph} of the pair $(\alpha,\beta)$ is the bipartite graph $\cG(\alpha\cup\beta)$
whose vertex set is $I\sqcup J$ and where there is an edge between two vertices $i\in I$ and $j\in J$ for every
intersection point between $\alpha_i$ and $\beta_j$.
\end{definition}

\noindent\textbf{Cylinder decompositions, bipartite graphs and harmonic functions}. Let $M$ be a flat surface having horizontal and vertical cylinder decompositions $\mathcal{H}=\{H_i\}_{i\in I}$ and $\mathcal{V}=\{V_j\}_{j\in J}$ such that each cylinder has modulus $\frac{1}{\lambda}$ for some $\lambda>0$. For every $i \in I$ let $\alpha_i$ be the
core curve of $H_i$ and for every $j \in J$ let $\beta_j$ be
the core curve of $V_j$.
Then $\alpha=\{\alpha_i\}_{i\in I}$ and $\beta=\{\beta_j\}_{j\in J}$ are multicurves whose union fills $M$. Let $\textbf{h}: I \cup J \to \R_{>0}$ be the function which to
an index associates the height of the corresponding cylinder.
Then $A\textbf{h} = \lambda \textbf{h}$ where $A$ is the
adjacency operator of the graph $\cG(\alpha \cup \beta)$, that is:
\begin{equation}
\label{eq:AdjacencyOperator}
(A\textbf{h})(v) := \sum_{w \sim v} \textbf{h}(w)
\end{equation}
where the sum above is taken over over edges $\{v,w\}$ having $v$ as one endpoint, that is, the summand  $\textbf{h}(w)$ appears as many times as there are edges between the vertices $v$ and $w$.
\begin{definition}
\label{def:LambdaHarmonicFunction}
Let $\cG = (V,E)$ be a graph with vertices of finite degree and $A: V^\R \to \R$ as in (\ref{eq:AdjacencyOperator}). A function $\textbf{h}: V \to \R$ that satisfies $A\textbf{h}=\lambda \textbf{h}$
is called a \emph{$\lambda$-harmonic function} of $\cG$.
\end{definition}

In summary: the existence of a horizontal and a vertical cylinder decomposition where each cylinders has  modulus $\frac{1}{\lambda}$ implies the existence of a \emph{positive} $\lambda$-harmonic function of configuration graph of the multicurves given by the core curves of the cylinders in the decomposition.

The idea to generalize Thurston-Veech's construction for infinite-type surfaces is to reverse this process: given a pair of multicurves $\alpha$ and $\beta$ whose union fills $S$, every positive $\lambda$-harmonic function of $\cG(\alpha\cup\beta)$ can be used to construct construct horizontal and  vertical cylinder decompositions of $S$ where all cylinders have the same modulus.

\begin{theorem}[Hooper-Thurson-Veech construction]
    \label{THM:HooperThurstonConstruction}
Let $S$ be an infinite-type surface. Suppose that there exist two multicurves $\alpha=\{\alpha_i\}_{i\in I}$ and $\beta=\{\beta_j\}_{j\in J}$ filling $S$ such that:
\begin{enumerate}
\item there is an uniform upper bound on the degree of the vertices of the configuration graph $\mathcal{G}(\alpha\cup\beta)$ and
\item every component of the complement of $S\setminus\alpha\cup\beta$ is a polygon with a finite number of sides\footnote{That is, each component of $S\setminus \alpha\cup\beta$ is a disc whose boundary consists of finitely man subarcs of curves in $\alpha\cup\beta$.}.
\end{enumerate}
Then there exists $\lambda_0\geq 2$ such that for every $\lambda\geq\lambda_0$ there exists a positive $\lambda$-harmonic function $\textbf{h}$ on $\cG(\alpha\cup\beta)$ which defines a flat surface structure $M=M(\alpha,\beta,\textbf{h})$ on $S$ admiting horizontal and vertical cylinder decompositions $\cH=\{H_i\}_{i\in I}$ and $\cV=\{V_j\}_{j\in J}$ where each cylinder has modulus $\frac{1}{\lambda}$. Moreover,  for every $i\in I$ and $j\in J$ the core curves of $H_i$ and $V_j$ are $\alpha_i$ and $\beta_j$, respectively.  In particular, we have (right) multitwists $T_\alpha$ and $T_\beta$ in $\Aff(M)$ which fix the boundary of each cylinder in $\cH$ and $\cV$, respectively. For each $\lambda\geq\lambda_0$, the derivatives of these multitwists define a representation $\rho:\langle T_\alpha,T_\beta\rangle\to\PSL(2,\mathbb{R})$ given by:
$$
T_\alpha\to \begin{pmatrix} 1 & \lambda \\ 0 & 1\end{pmatrix},\hspace{5mm} T_\beta\to \begin{pmatrix} 1 & 0 \\ -\lambda & 1\end{pmatrix}.
$$
\end{theorem}

\begin{remark}
    \label{Remark:CreditsHTVConstruction}
Theorem~\ref{THM:HooperThurstonConstruction} is a particular case of a more general version of Hooper-Thurston-Veech's construction due to V. Delecroix and the second author whose final form was achieved after discussions with the first author, see~\cite{DHV19}. We do not need this more general version for the proof of our main results. On the other hand, the second assumption on the multicurves in~Theorem~\ref{THM:HooperThurstonConstruction} makes the proof simpler that in the general case and for this reason we decided to sketch it.
Many of the key ideas in the proof of the result above and its general version appear already in the work of P. Hooper~\cite{Hooper15}. The main difference is that P. Hooper starts with a bipartite infinite graph with uniformly bounded valence and then, using a positive harmonic function $h$ on that graph, constructs a translation surface. We take as input an infinite-type topological surface $S$ and a pair of filling multicurves to construct a flat surface structure on $S$, which is a not \emph{a priori} a translation surface structure.
\end{remark}


\noindent\textbf{Proof of Theorem~\ref{THM:HooperThurstonConstruction}}. 
The union $\alpha\cup\beta$ of the multicurves $\alpha$ and $\beta$ defines a graph embedded in $S$: the vertices are points in $\bigcup_{(i,j)\in I\times J}\alpha_i\cap\beta_j$ and edges are the segments forming the connected components of $\alpha\cup\beta\setminus \bigcup_{(i,j)\in I\times J}\alpha_i\cap\beta_j$. Abusing notation we write $\alpha\cup\beta$ to refer to this graph. It is important not to confuse the (geometric) graph $\alpha\cup\beta$ with the (abstract) configuration graph $\cG(\alpha\cup\beta)$. To define the flat structure $M$ on $S$ we consider a dual graph $(\alpha\cup\beta)^*$ defined as follows. If $S$ had no punctures then $(\alpha\cup\beta)^*$ is just the dual graph of $\alpha\cup\beta$. If $S$ has punctures\footnote{We think of punctures also as isolated ends or points at infinity.} we make the following convention to define the vertices of $(\alpha\cup\beta)^*$: for every connected component $D$ of $S\setminus \alpha\cup\beta$ homeomorphic to a disc choose a unique point $v_D$ inside the connected component. If $D$ is a punctured disc, then choose $v_D$ to be the puncture.


\begin{figure}[!ht]
\begin{center}
\includegraphics{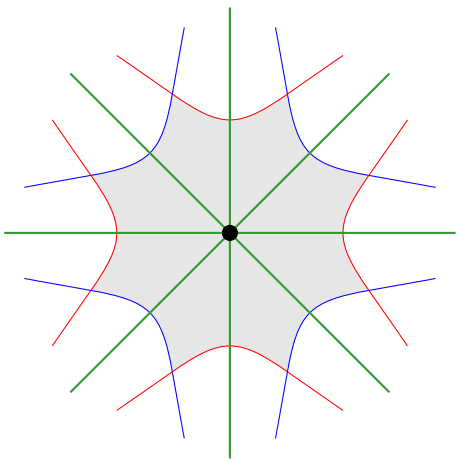}
\end{center}
\caption{The graph $(\alpha\cup\beta)^*$.}
\label{fig:DualGraph}
\end{figure}

The points $v_D$ chosen above are the vertices of $(\alpha\cup\beta)^*$. Vertices in this graph are joined by an edge in $S$ if the closures of the corresponding connected components of $S\setminus \alpha\cup\beta$ share an edge of $\alpha\cup\beta$. Edges are chosen to be pairwise disjoint. Remark that $(\alpha\cup\beta)^*$ might have loops. See Figure~\ref{fig:DualGraph}.


Given that $\alpha\cup\beta$ fills, every connected component $S\setminus (\alpha\cup\beta)^*$ is a topological quadrilateral which contains a unique vertex of $\alpha\cup\beta$. Hence there is a well defined bijection between edges in the abstract graph $\mathcal{G}(\alpha\cup\beta)$ and the set of these quadrilaterals. This way, for every edge $e\in E(\mathcal{G}(\alpha\cup\beta))$ we denote by $R_e$ the closure in $S$ of the corresponding topological quadrilateral with the convention to add to $R_e$ vertices $v_D$ corresponding to punctures
in $S$.

Note that there are only two sides of $R_e$ intersecting the multicurve $\alpha$, which henceforth are called \emph{vertical sides}. The other two sides are in consequence called \emph{horizontal}, see Figure~\ref{fig:MakingRectanglesEuclidean}.

We now build a flat surface structure on $S$ by declaring the topological
quadrilaterals $R_e$ of the dual graph $(\alpha \cup \beta)^*$ to be Euclidean
rectangles. Given that there is a uniform upper bound on the degree of the vertices of the configuration graph $\cG(\alpha \cup \beta)$ there exists\footnote{For a more detailed discussion on $\lambda$-harmonic functions we recommend Appendix C in~\cite{Hooper15} and reference therein.} $\lambda_0\geq 2$ such that for every $\lambda\geq\lambda_0$ there exists a positive $\lambda$-harmonic function $\textbf{h}: \cG(\alpha \cup \beta) \to \R_{>0}$. We use this function to define compatible heights of horizontal and vertical cylinders.
More precisely, let us define the maps:
$$
p_\alpha:E(\cG(\alpha\cup\beta)) \to V(\cG(\alpha\cup\beta))
\qquad \text{and} \qquad
p_\beta: E(\cG(\alpha\cup\beta)) \to V(\cG(\alpha\cup\beta))
$$
which to an edge $e$ of the configuration graph $\cG(\alpha\cup\beta)$ associate
its endpoints $p_\alpha(e)$ in $I$ and $p_\beta(e)$ in $J$. The desired flat structure is defined by declaring\footnote{For a formal description on how to identify $R_e$ with $[0,\textbf{h}\circ p_\beta(e)] \times [0,\textbf{h}\circ p_\alpha(e)]$ we refer the reader to~\cite{DHV19}.} $R_e$ to be the rectangle $[0,\textbf{h}\circ p_\beta(e)] \times [0,\textbf{h}\circ p_\alpha(e)]$, see Figure~\ref{fig:MakingRectanglesEuclidean}.


\begin{figure}[!ht]
\begin{center}
\includegraphics{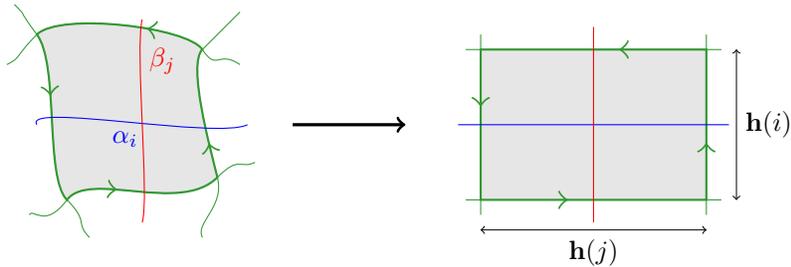}
\end{center}
\caption{Transforming topological rectangles into Euclidean rectangles.}
\label{fig:MakingRectanglesEuclidean}
\end{figure}

We denote the resulting flat surface $M(\alpha, \beta, \textbf{h})$. Remark that by contruction a vertex $v_D$ of the dual graph $(\alpha\cup\beta)^*$ of valence $k$ defines a conical singularity of angle $\frac{\pi k}{2}$ in the metric completion of $M(\alpha, \beta, \textbf{h})$. Given that $k$ is always an even number we have that $M(\alpha, \beta, \textbf{h})$ is a half-translation surface (\emph{i.e.} given by a quadratic differential) when $k=2(2n-1)$ for some $n\in\mathbb{Z}_{\geq 1}$.\\

Now, for every $i\in I$, the curve $\alpha_i$ is the core curve of the horizontal
cylinder $H_i := \cup_{e\in p_{\alpha}^{-1}(i)} R_e$. Because $\textbf{h}$ is $\lambda$-harmonic we have
$$
\sum_{e\in p_\alpha^{-1}(i)}\textbf{h}(p_{\beta}(e))=\sum_{j\sim i}\textbf{h}(j)=\lambda \textbf{h}(i).
$$
This equations say that the circumference $\sum_{e\in p_\alpha^{-1}(i)}\textbf{h}(p_{\beta}(e))$ of
$H_i$ is $\lambda$ times its height $\textbf{h}(i)$, hence the modulus of $H_i$
is equal to $\frac{1}{\lambda}$. The same computation with $\beta_j$ shows that the vertical cylinders
$V_j := \cup_{e\in p_{\beta}^{-1}(j)} R_e$ have core curve $\beta_j$
and modulus $\frac{1}{\lambda}$.\qed

\begin{remark}
As said before, Hooper-Thurston-Veech construction can be applied to more general pairs of multicurves $\alpha,\beta$. Consider for example the case in the Loch Ness monster depicted in Figure~\ref{Fig:MulticurvesInfiniteStaircase}: here the graph $\cG(\alpha\cup\beta)$ has finite valence but there exist four connected componets $\{C_i\}_{i=1}^4$ of $S\setminus\alpha\cup\beta$ which are infinite polygons, that is, whose boundary is formed by infinitely many segments belonging to curves in $\alpha$ and in $\beta$. In this situation the convention is to consider vertices in the dual graph $(\alpha\cup\beta)^*$ of infinite degree as points at infinity (that is, not in $S$). With this convention the Hooper-Thuston-Veech construction produces a translation surface structure on $S$, because each $\partial C_i$ is connected. In Figure~\ref{Fig:MulticurvesInfiniteStaircase} we illustrate the case of a $2$-harmonic function; the resulting flat surface is a translation surface known as the infinite staircase.
\\
\begin{figure}[!ht]
\begin{center}
\begin{minipage}{.45\textwidth}
\begin{center}
\includegraphics[scale=.5]{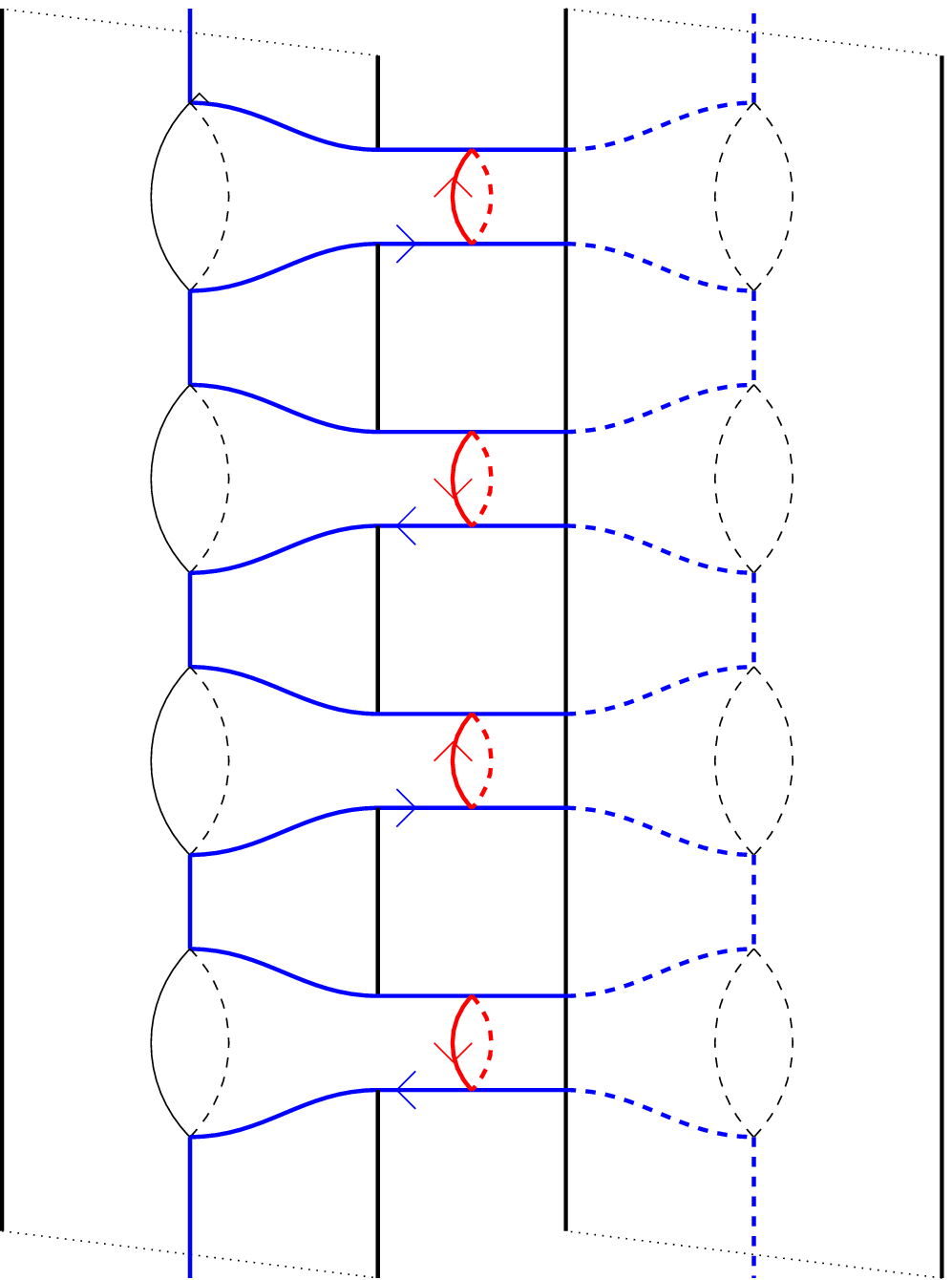}
\end{center}
\subcaption{Two oriented multicurves $\alpha$ (in blue) and $\beta$ (in red) in the Loch Ness monster for which the Hooper-Thurston-Veech construction produces the infinite staircase.}
\end{minipage}
\begin{minipage}{.45\textwidth}
\begin{center}
\includegraphics[scale=1]{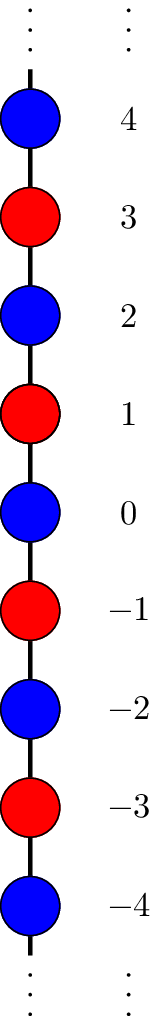}
\end{center}
\subcaption{The graph $\cG(\alpha\cup\beta)$ }
\end{minipage}
\begin{minipage}{.45\textwidth}
\begin{center}
\includegraphics[scale=1.2]{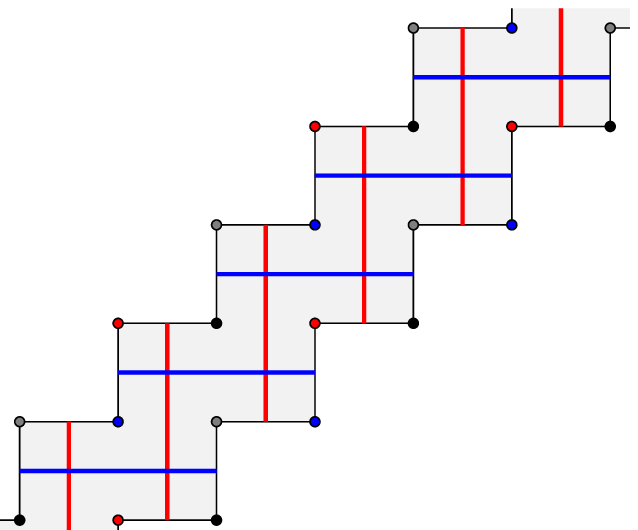}
\end{center}
\end{minipage}
\end{center}
\caption{The infinite staircase as a Hooper-Thurson-Veech surface.}
\label{Fig:MulticurvesInfiniteStaircase}
\end{figure}




\emph{Proof of Theorem~\ref{COR:FlatAmbivalentHomeos}}. Let $\lambda\geq 2$ and consider the subgroup of $\SL(2,\R)$:
\begin{equation}
\label{eq:GroupGLambda}
G_{\lambda}:=\langle\begin{psmallmatrix}1 & \lambda\\0 & 1\end{psmallmatrix},\begin{psmallmatrix}1 & 0\\ -\lambda & 1\end{psmallmatrix}\rangle
\end{equation}
This group is free and its elements are matrices of the form $\begin{psmallmatrix}1+k_{11}\lambda^2 & k_{12}\lambda\\ k_{21}\lambda & 1+k_{22}\lambda^2\end{psmallmatrix}$, $k_{ij}\in\Z$, such that the determinant is 1 and $|\frac{1+k_{11}\lambda^2}{k_{12}\lambda}|$ does not belong to the interval $(t^{-1},t)$, where $t=\frac{1}{2}(\lambda+\sqrt{\lambda^2-4})$, see~\cite{Brenner55}. On the other hand, since
$$
\{-\frac{\lambda}{2}<\Im(z)\leq\frac{\lambda}{2}\}\cap\{|z+\frac{1}{2\lambda}|>\frac{1}{2\lambda}\}\cap\{|z-\frac{1}{2\lambda}|\geq\frac{1}{2\lambda}\}\subset\mathbb{H}^2
$$
is a fundamental domain in the hyperbolic plane for $G_\lambda$, this group has no elliptic elements.  Morevoer, if $\lambda > 2$ there are only two conjugacy classes of parabolics (correspoding to the generators of $G_\lambda$) and if $\lambda=2$ then $\begin{psmallmatrix}1 & \lambda\\0 & 1\end{psmallmatrix}\begin{psmallmatrix}1 & 0\\ -\lambda & 1\end{psmallmatrix}=\begin{psmallmatrix}1-\lambda^2 & \lambda\\-\lambda & 1\end{psmallmatrix}$ and $\begin{psmallmatrix}1 & -\lambda\\0 & 1\end{psmallmatrix}\begin{psmallmatrix}1 & 0\\ \lambda & 1\end{psmallmatrix}=\begin{psmallmatrix}1-\lambda^2 & -\lambda\\ \lambda & 1\end{psmallmatrix}$ determine, together with the generators of $G_\lambda$, the only 4 conjugacy classes of parabolics in $G_\lambda$.  Remark that $\begin{psmallmatrix}1-\lambda^2 & \lambda\\-\lambda & 1\end{psmallmatrix}$ and $\begin{psmallmatrix}1-\lambda^2 & -\lambda\\ \lambda & 1\end{psmallmatrix}$ are hyperbolic if $\lambda>2$.

If $\alpha$ and $\beta$ are the multicurves depicted in Figure~\ref{Fig:MulticurvesInfiniteStaircase} (A) on the Loch Ness monster then $\cG(\alpha\cup\beta)$ is
the infinite bipartite graph on Figure~\ref{Fig:MulticurvesInfiniteStaircase} (B). Let us index the vertices of this graph by the integers as in the Figure so that $\textbf{h}_2(n)=1$, for all $\hspace{1mm}n\in\Z$, is a positive $2$-harmonic function on $\cG(\alpha\cup\beta)$. If $\lambda>2$ and  $r_+=\frac{\lambda+\sqrt{\lambda^2-4}}{2}$ the positive function $\textbf{h}_{\lambda}(n)=r_+^n,\hspace{1mm}n\in\Z$ is $\lambda$-harmonic on $\cG(\alpha\cup\beta)$. The desired family of translation surfaces $\{M_\lambda\}_{\lambda\in[2,+\infty)}$,  is obtained by applying Hooper-Thuston-Veech's construction to the multicurves $\alpha$, $\beta$ and the family of positive $\lambda$-harmonic functions $\{\textbf{h}_{\lambda}\}_{\lambda\in[2,+\infty)}$. The desired class $f\in\MCG(S)$ is given by the product of (right) multitwist $T_\alpha T_\beta$. No positive power of $f$ fixes an isotopy class of simple closed curve in $S$ because on one hand if $\lambda=2$ the translation flow on eigendirection of the parabolic matrix $D\tau_2^{-1}(f)$ decomposes $M_2$ in two strips and, on the other if $\lambda>2$, then $D\tau_\lambda^{-1}(f)$ is hyperbolic.
\qed

\begin{remark}
There are only 3 infinite graphs which admit $2$-harmonic functions, these are depicted in Figure~\ref{fig:2harmonicgraphs} together with their correspoding positive $2$-harmonic functions (which are unique up to rescaling). None of them comes from a pair of multicurves satifying (2) in Theorem~\ref{THM:HooperThurstonConstruction}.
\end{remark}

\begin{figure}
\begin{center}
\includegraphics[scale=0.9]{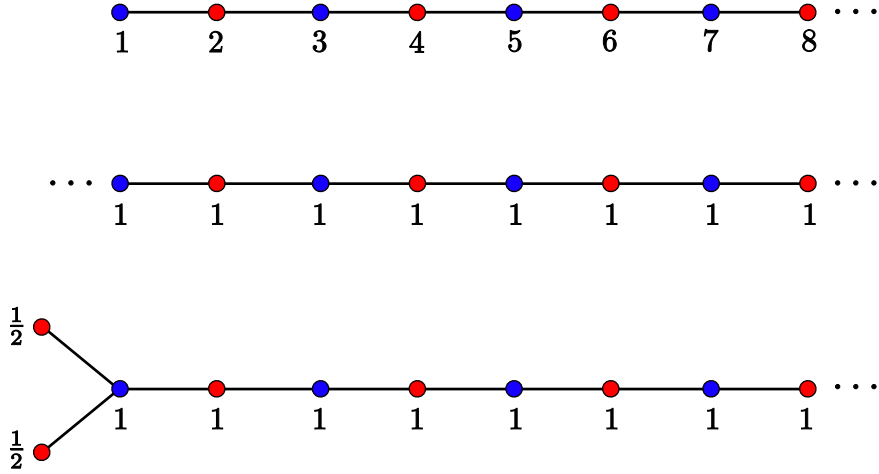}
\end{center}
\caption{Graphs with $2$-harmonic functions.}
\label{fig:2harmonicgraphs}
\end{figure}


\end{remark}


\textbf{Renormalizable directions}. The main results on Hooper's work~\cite{Hooper15} deal with the dynamical properties of the translation flow in \emph{renormalizable directions}.
\begin{definition}
    \label{DEF:RenormalizableDirection}
Consider the action of $G_\lambda$ as defined in (\ref{eq:GroupGLambda}) by homographies on the real projective line $\mathbb{RP}^1$. We say that a direction $\theta\in\R/2\pi\Z$ is $\lambda$-renormalizable if its projectivization lies in the limit set of $G_\lambda$ and is not an eigendirection of any matrix conjugated in $G_\lambda$ to a matrix of the form:
$$
\begin{pmatrix} 1 & \lambda \\ 0 & 1\end{pmatrix},\hspace{2mm}\begin{pmatrix} 1 & 0 \\ \lambda & 1\end{pmatrix} \hspace{2mm} or \hspace{2mm} \begin{pmatrix} 1 & 0 \\ -\lambda & 1\end{pmatrix}\cdot\begin{pmatrix} 1 & \lambda \\ 0 & 1\end{pmatrix}
$$
\end{definition}

We use two of Hooper's results in the proof of Theorem~\ref{THM:HooperThurstonConstruction}. Recall that in Hooper's work one takes as input an infinite bipartite graph and a positive $\lambda$-harmonic function on this graph to produce a translation surface.

\begin{theorem}[Theorem 6.2,~\cite{Hooper15}]
    \label{THM:NoSaddleConnections}
Let $M$ be a translation surface obtained from an infinite bipartite graph as in~[\emph{Ibid}.] using a positive $\lambda$-harmonic function and let $\theta$ be a $\lambda$-renormalizable direction. Then
the translation flow $F_\theta^t$ on $M$ does not have saddle connections.
\end{theorem}

\begin{theorem}[Theorem 6.4,~\cite{Hooper15}]
    \label{THM:ConservativeFlow}
Let $M$ be a translation surface obtained from an infinite bipartite graph as in~[\emph{Ibid}.] using a positive $\lambda$-harmonic function and let $\theta$ be a $\lambda$-renormalizable direction. Then
the translation flow $F_\theta^t$ is conservative, that is, given $A\subset M$ of positive measure and any $T>0$, for Lebesgue almost every $x\in M$ there is a $t>T$ such that $F_\theta^t(x)\in A$
\end{theorem}



\section{Proof of results}
    \label{Section:ProofResults}

\subsection{Proof of Theorem~\ref{THM:Loxodromics}} The proof is divided in two parts. In the first part we use the Hooper-Thurston-Veech construction (see Section \ref{Subsection:HTVConstruction}) to find two transverse measured $f$-invariant foliations $\mathcal{F}^u$ and $\mathcal{F}^s$ on $S$ for which $p$ is a singular point and for which each foliation has $m$ separatrices based at $p$. We prove that each separatrix based at $p$ is dense in $S$. Then, we consider a hyperbolic metric on $S$ of the first kind (allowing us to talk about the completed ray graph $R(S;p)$). We stretch each separatrix of $\mathcal{F}^u$ and $\mathcal{F}^s$ based at $p$ to a geodesic with respect to this metric. This defines two sets
 $\Gamma^+$ and $\Gamma^-$ of geodesics, each having cardinality $m$. In the second part of the proof, we show that $\Gamma^+$ and $\Gamma^-$ are the only cliques of high-filling rays fixed by $f$ in the Gromov boundary of the loop graph.

\emph{Flat structures}. We use the Hooper-Thurston-Veech construction (Section~\ref{Subsection:HTVConstruction}) for this part of the proof. Let $\alpha$ and $\beta$ be two multicurves satisfying the hypothesis of Theorem~\ref{THM:Loxodromics}. Fix $\mathbf{h}:\mathcal{G}(\alpha\cup\beta)\to\mathbb{R}_{>0}$ a positive $\lambda$-harmonic function on the configuration graph $\mathcal{G}(\alpha\cup\beta)$. Let $M=M(\alpha,\beta,\mathbf{h})$ be the flat structure on $S$ given by the Hooper-Thurston-Veech construction and
$$
\rho:\langle T_\alpha,T_\beta\rangle\to\PSL(2,\mathbb{R})
$$
the corresponding presentation. Here, we have chosen $p$ as one of the vertices of the dual graph $(\alpha\cup\beta)^*$ (see the proof of Theorem~\ref{THM:HooperThurstonConstruction}) and therefore it makes sense to consider the classes that the affine multitwists $T_\alpha$, $T_\beta$ define in $\MCG(S;p)$. We abuse notation and denote also by $T_\alpha$, $T_\beta$ these classes.

The eigenspaces of the hyperbolic matrix $\rho(f)$ define two transverse ($f$-invariant) measured foliations $(\mathcal{F}^u,\mu_u)$ and $(\mathcal{F}^s,\mu_s)$ (for unstable and stable, respectively) on $M$ . Moreover, we have that $f\cdot(\mathcal{F}^u,\mu_u)=(\mathcal{F}^u,\eta\mu_u)$ and $f\cdot(\mathcal{F}^s,\mu_s)=(\mathcal{F}^s,\eta^{-1}\mu_u)$, where $\eta>1$ is (up to sign) an eingenvalue of $\rho(f)$.
For simplicity we abbreviate the notation for these foliations by $\mathcal{F}^u$ and $\mathcal{F}^s$.


\emph{The set $\mathfrak{V}$}. Recall that $M=M(\alpha,\beta,\textbf{h})$ is constructed by glueing a family of rectangles $\{R_e\}_{e\in E}$, where $E=E(\mathcal{G}(\alpha\cup\beta))$ is the set of edges of the configuration graph $\mathcal{G}(\alpha\cup\beta)$, along their edges using translations and half-translations. By the way the Hooper-Thurston-Veech construction is carried out, sometimes the corners of these rectangles are not part of the surface $S$: this is the case when there are connected components of $S\setminus\alpha\cup\beta$ which are punctured discs. However, every corner of a rectangle $\{R_e\}_{e\in E}$ belongs to the metric completion $\widehat{M}$ of $M$ (w.r.t. the natural flat metric). We define $\mathfrak{V}\subset\widehat{M}$ to be the set of points that are corners of rectangles in $\{R_e\}_{e\in E}$ (after glueings). Remark that since
all connected components of $S\setminus \alpha\cup\beta$ are (topological) polygons with an uniformly bounded number of sides, points in $\mathfrak{V}$ are regular points or conical singularities of $\widehat{M}$ whose total angle is uniformly bounded. Moreover the set $\Fix(f)$ of fixed points of the continuous extension of $f$ to $\widehat{M}$ contains $\mathfrak{V}$. Indeed, if $\cH=\{H_i\}$ and $\cV=\{V_j\}$ denote the horizontal and vertical (maximal) cylinder decompositions of $M$, then $\mathfrak{V}=\cup_{i\neq j} (\partial H_i\cap \partial V_j)$, where the boundary of each cylinder is taken in the metric completion $\widehat{M}$. The claim follows from the fact that for every $i\in I$ and $j\in J$, $T_\alpha$ fixes $\partial H_i$ and $T_\beta$ fixes $\partial V_j$.

For each $q\in \mathfrak{V}$ we denote by $\Sep_q(*)$ the set of leaves of $*\in\{\mathcal{F}^u,\mathcal{F}^s\}$ based at $q$.
We call such a leaf a \emph{separatrix based at $q$}. Remark that if the total angle of the flat structure $M$ around $q$ is $k\pi$ then $|\Sep_q(\mathcal{F}^u)|=|\Sep_q(\mathcal{F}^s)|=k$. The following fact is essential for the second part of the proof.


\begin{proposition}
	\label{ref:DensitySeparatrices}
Let $q\in \mathfrak{V}$. Then any separatrix in $\Sep_q(\mathcal{F}^u)\cup\Sep_q(\mathcal{F}^s)$ is dense in $M$.
\end{proposition}
\begin{proof}
We consider first the case when $M$ is a translation surface. At the end of the proof we deal with the case when $M$ is a half-translation surface.

We show that any separatrix in $\Sep_q(\mathcal{F}^u)$ is dense. The arguments for separatrices in $\Sep_q(\mathcal{F}^s)$ are analogous.

\emph{Claim}: $\cup_{q\in\mathfrak{V}}\Sep_q(\mathcal{F}^u)$, the union of all separatrices of $\mathcal{F}^u$, is dense in $M$. To prove this claim we strongly use the work of Hooper ~\cite{Hooper15}\footnote{Hooper only deals with the case when $M$ is a translation surface. This is why when $M$ is a half-translation surface we consider its orientation double cover.}. In particular, we use the fact that leaves in $\mathcal{F}^u$ are parallel to a renormalizable direction, see Definition~\ref{DEF:RenormalizableDirection}. We proceed by contradiction by assuming that the complement of the closure of $\cup_{q\in V}\Sep_q(\mathcal{F}^u)$ in $\widehat{M}$ is non-empty. Let $U$ be a connected component of this complement. Then $U$ is $\mathcal{F}^u$-invariant. If $U$ contains a closed leaf of $\mathcal{F}^u$ then it has to be a cylinder, but this cannot happen because there are no saddle connections parallel to renormalizable directions, see Theorem 6.2 in~\cite{Hooper15}. Then $U$ contains a transversal to $\mathcal{F}^u$ to which leaves never return. In other words, $U$ contains an infinite strip, \emph{i.e.} a set which (up to rotation) is isometric to $(a,b)\times\mathbb{R}$ for some $a<b$. This is impossible since the translation flow on $M$ in a renormalizable direction is conservative\footnote{The translation flow $F_\theta^t$ is called conservative if given $A\subset M$ of positive measure and any $T>0$, for Lebesgue almost every $x\in M$ there is a $t>T$ such that $F_\theta^t(x)\in A$.}, see Theorem 6.4 in~\cite{Hooper15}. The claim follows.

We strongly recommend that the reader uses Figure~\ref{fig:Separatricesdensas} as a guide for the next paragraphs.

\begin{figure}[!ht]
\begin{center}
\includegraphics[scale=0.8]{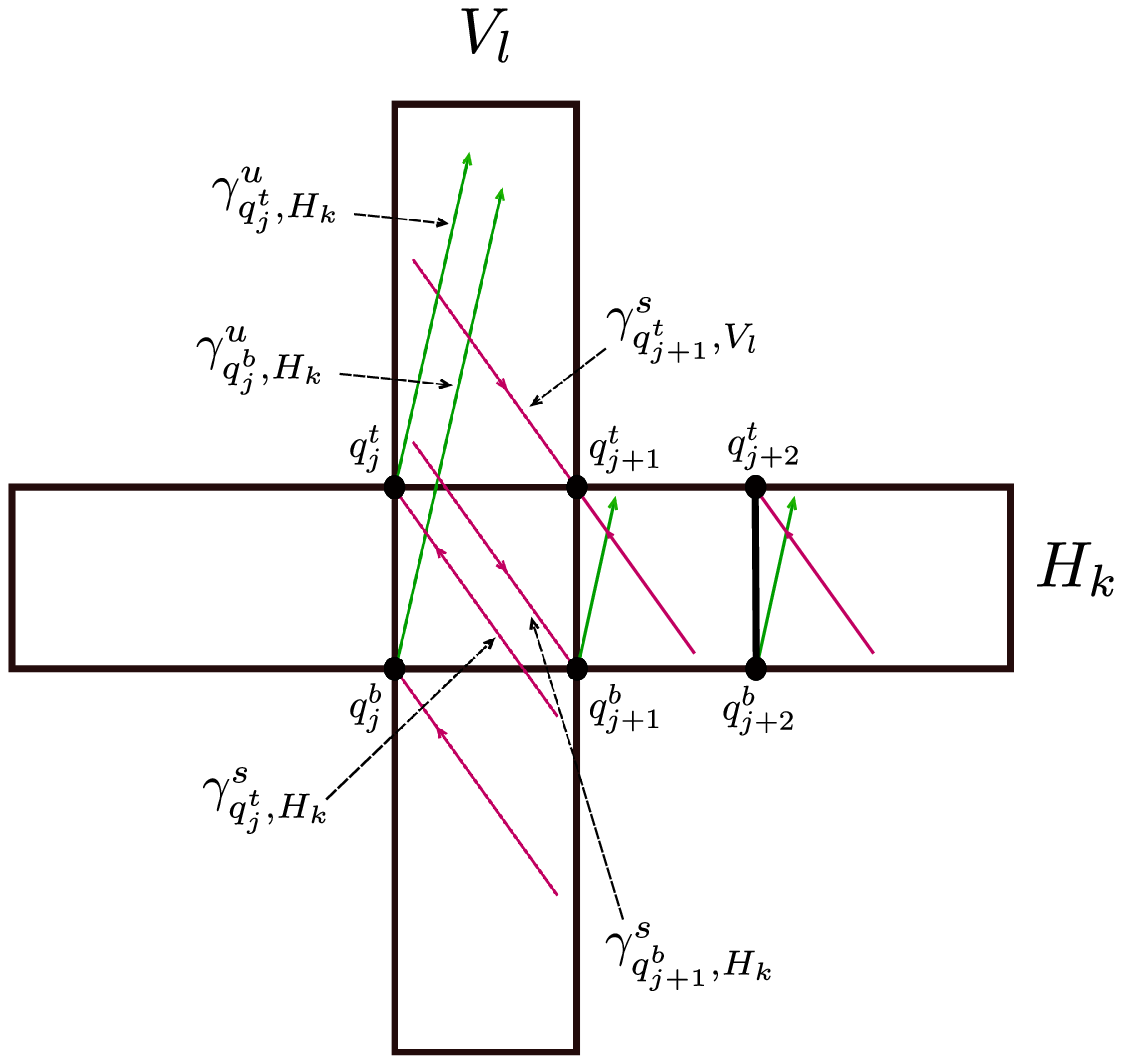}
\end{center}
\caption{}
\label{fig:Separatricesdensas}
\end{figure}

Henceforth if $\gamma$ is a separatrix of $\mathcal{F}^u$, we denote by $\gamma(t)$, $t>0$ the parametrization for which $\lim_{t\to 0}\gamma(t)\in \mathfrak{V} $ and such that $|\gamma'|=1$ (w.r.t. to the flat metric on $M$).

For each horizontal cylinder $ H_k$ in $M$ and $\xi\in \mathfrak{V}\cap \partial H_k$ we denote by $\gamma^u_{\xi, H_k}\subset M$ (respectively $\gamma^s_{\xi, H_k}$) the unique separatrix of
$\mathcal{F}^u$ (respect. of $\mathcal{F}^s$) based at $\xi$ within $ H_k$, that is, for which $\gamma^u_{\xi, H_k}(t)\in H_k$ for all $t$ in a small neighbourhood of $0$.
For a vertical cylinder $ V _l$, $\gamma^u_{\xi, V _l}$ and $\gamma^s_{\xi, V _l}$ are defined in a similar way. Let $\mathfrak{V}^b( H_k)$ and $\mathfrak{V}^t( H_k)$ denote the points in $ \mathfrak{V} \cap H_k$ in the bottom
and in the top\footnote{We pull back the standard orientation of the Euclidean plane to $M$ to make sense of the east-west and bottom-top sides of a cylinder.} connected component of $\partial H_k$ respectively; and for any vertical cylinder $ V _l$ let $\mathfrak{V}^e( V _l)$ and $\mathfrak{V}^w( V _l)$ denote the points in $ \mathfrak{V} \cap V _l$ in the east and west connected component of $\partial V _l$ respectively.

Without loss of generality we suppose that $q\in V^b( H_k)\cap V^e( V _l)$. We denote by $\omega(\gamma_{q, H_k}^u)$ the $\omega$-limit set of $\gamma_{q, H_k}^u$.

\emph{Claim}: the union of all separatrices of $\mathcal{F}^u$ based at points in $\partial H_k\cup\partial V_l$ within $H_k$ and $V_l$ respectively
\begin{equation}
  \label{EQ:HorizontalSeparatrices}
\left(\bigcup_{\xi\in \mathfrak{V}\cap\partial H_k}\gamma^u_{\xi, H_k}\right)\cup\left(\bigcup_{\xi\in \mathfrak{V}\cap\partial V_l}\gamma^u_{\xi, V_l}\right)
\end{equation}
is contained in $\omega(\gamma_{q, H_k}^u)$.

\emph{Proof of claim}. Remark that since $H_k$ is tiled by rectangles corresponding to points of intersection of the core curve $\alpha_k$ with curves in $\beta$, $|\mathfrak{V}^b( H_k)|=|\mathfrak{V}^t( H_k)|$ and for each $\xi\in \mathfrak{V}^b( H_k)$ there is exactly one point in $\mathfrak{V}^t( H_k)$ just above. Hence, using the east-west orientation of $ H_k$, we can order the elements of $\mathfrak{V}^b( H_k)\cup \mathfrak{V}^t( H_k)$ cyclically: we write $\mathfrak{V}^b( H_k)=\{q_j^b\}_{j\in\mathbb{Z}/N\mathbb{Z}}$, $\mathfrak{V}^t( H_k)=\{q_j^t\}_{j\in\mathbb{Z}/N\mathbb{Z}}$ for some $N\geq 1$. The sets $\mathfrak{V}^e( V_l)=\{q_j^e\}_{j\in\mathbb{Z}/M\mathbb{Z}}$, $\mathfrak{V}^w(V_l)=\{q_j^w\}_{j\in\mathbb{Z}/M\mathbb{Z}}$, for some $M\geq 1$, are defined in a similar way.

We suppose that the labeling is such that above $q_j^b$ lies $q_j^t$ for all $j\in\mathbb{Z}/N\mathbb{Z}$, and that $q=q_0^b=q_0^e$. Recall that $DT_\alpha=\begin{psmallmatrix} 1 & \lambda \\ 0 & 1\end{psmallmatrix}$, $DT_\beta^{-1}=\begin{psmallmatrix} 1 & 0 \\ \lambda & 1\end{psmallmatrix}$, hence $D f=\begin{psmallmatrix} a & b \\ c & d\end{psmallmatrix}$, with $a,b,c,d\in\mathbb{R}_{>0}$.
In particular $D f$ sends the positive quadrant $\mathbb{R}_{x\geq 0,y\geq 0}$ into itself. If we suppose, without loss of generality, that the unstable eingenspace of $D f$ (without its zero) lies in the interior of $\mathbb{R}_{x\geq 0,y\geq 0}\cup\mathbb{R}_{x\leq 0,y\leq 0}$, then the stable eigenspace of $D f$ (without its zero) has to lie in the interior of $\mathbb{R}_{x\geq 0,y\leq 0}\cup\mathbb{R}_{x\leq 0,y\geq 0}$.
Hence, for every $j\in\mathbb{Z}/N\mathbb{Z}$ we have that $\gamma^u_{q_j^b,H_k}$ intersects $\gamma^s_{\xi,H_k}$, for every $\xi\in\{q_j^t,q_{j+1}^b\}$ and $\gamma^s_{q_{j+1}^t,V_{l'}}$, where $V_{l'}$ is the vertical cylinder intersecting $H_k$ and having $\{q_j^b,q_j^t,q_{j+1}^b,q_{j+1}^t\}$ in its boundary\footnote{Remark that in $M$ these points need not to be all different from each other. For example $q_j^b=q_j^t$ and $q_{j+1}^b=q_{j+1}^t$ if the core curve of $V_{l'}$ only intersects the core curve of $H_k$. In any case the claims remain valid.}. From Figure~\ref{fig:Separatricesdensas} we can see that some of these points of intersection are actually in $H_k\cup V_{l'}$.
By applying repeatedly $f$ to all these points of intersection of separatrices we obtain that $\xi\in\omega(\gamma_{q_j^b}^u,H_k)$ for every $j\in\mathbb{Z}/N\mathbb{Z}$ and $\xi\in\{q_{j+1}^b,q_j^t,q_{j+1}^t\}$. This implies that $\gamma^u_{\xi,H_k}\subset\omega(\gamma_{q_j^b,H_k}^u)$ for every $j\in\mathbb{Z}/N\mathbb{Z}$ and $\xi\in\{q_{j+1}^b,q_j^t,q_{j+1}^t\}$.
In particular, we get that $\omega(\gamma_{q=q_0^b,H_k}^u)$ contains $\gamma^u_{\xi,H_k}$ for every  $\xi\in\{q_1^b,q_1^t,q_0^t\}$. As a consequence, we have that $\omega(\gamma_{q,H_k}^u)$ contains\footnote{Here we are using the following general principle: if $\gamma_1,\gamma_2$ are trayectories of a vector field on a translation surface and $\gamma_1$ is contained in $\omega(\gamma_2)$, then $\omega(\gamma_1)\subset\omega(\gamma_2)$.
} $\omega(\gamma^u_{q_1^b,H_k})$, which in turn  contains $\{\gamma^u_{q_2^b,H_k},\gamma^u_{q_1^t,H_k},\gamma^u_{q_2^t,H_k}\}$. Proceeding inductively we get that $\omega(\gamma_{q,H_k}^u)$ contains $$
\bigcup_{\xi\in \mathfrak{V}\cap\partial H_k}\gamma^u_{\xi, H_k}.
$$

The positivity of the matrix $Df$ and the fact that its unstable eigenspace lies in  $\mathbb{R}_{x\geq 0,y\geq 0}\cup\mathbb{R}_{x\leq 0,y\leq 0}$ also imply that for every $j\in\mathbb{Z}/M\mathbb{Z}$ the separatrix $\gamma_{q_j^e,V_l}^u$ intersects $\gamma_{\xi,V_l}^s$ for $\xi\in\{q_j^w,q_{j+1}^e,q_{j+1}^w\}$. From here on, the logic to show that $\omega(\gamma_{q,H_k}^u)$ contains $\bigcup_{\xi\in \mathfrak{V}\cap\partial V_l}\gamma^u_{\xi, V_l}$ is the same as the one presented in the preeceding paragraph and the claim follows.

The arguments in the proof of the preceding claim are local so they can be used to show that:
\begin{itemize}
\item For every $j\in\mathbb{Z}/N\mathbb{Z}$, the limit set $\omega(\gamma_{q_j^b,H_k}^u)$ contains all separatrices:
$$
\left(\bigcup_{\xi\in \mathfrak{V}\cap\partial H_k}\gamma^u_{\xi, H_k}\right)\cup\left(\bigcup_{\xi\in \mathfrak{V}\cap\partial V_{l'}}\gamma^u_{\xi, V_{l'}}\right)
$$
where $V_{l'}$ is such that $q_j^b=\partial H_k\cap\partial V_{l'}$.
\item For every $j\in\mathbb{Z}/M\mathbb{Z}$, the limit set $\omega(\gamma_{q_j^e,H_k}^u)$ contains all separatrices:
$$
\left(\bigcup_{\xi\in \mathfrak{V}\cap\partial V_l}\gamma^u_{\xi, V_l}\right)\cup\left(\bigcup_{\xi\in \mathfrak{V}\cap\partial H_{k'}}\gamma^u_{\xi, H_{k'}}\right)
$$
where $H_{k'}$ is a horizontal cylinder such that $q_j^e=\partial V_l\cap\partial H_{k'}$.
\end{itemize}
If we now denote by $\alpha_k\in\alpha$ the core curve of $H_k$ then the preceding discussion can be summarized as follows:
$\omega(\gamma_{q,H_k}^u)$ contains all separatrices of $\mathcal{F}^u$ based at points in the boundary of cylinders (and stemming within those cylinders) whose core curves belong to the link of $\alpha_k$ in the configuration graph $\mathcal{G}(\alpha\cup\beta)$; moreover, if $\beta_{l}\in{\rm link}(\alpha_k)$ then $\omega(\gamma_{q,H_k}^u)$ contains all separatrices of $\mathcal{F}^u$ based at points in the boundary of cylinders (and stemming within those cylinders) whose core curves belong to ${\rm link}(\beta_{l})$. This way we can extend the arguments above to the whole configuration graph $\mathcal{G}(\alpha\cup\beta)$ to conclude that $\omega(\gamma_{q,H_k}^u)$ contains $\cup_{q\in \mathfrak{V}}\Sep_q(\mathcal{F}^u)$. Since the later is dense in $M$ we conclude that $\gamma_{q,H_k}^u$ is dense in $M$.

We now suppose that $M$ is a half-translation surface (\emph{i.e.} given by a quadratic differential). Let $\pi:\widetilde{M}\to M$ the orientation double cover of $M$.

We claim that for every horizontal cylinder $H_i$ in $M$ the lift  $\pi^{-1}(H_i)$ is formed by two disjoint isometric copies $\widetilde{H}_{i_1}$, $\widetilde{H}_{i_2}$ of $H_i$ and these are maximal horizontal cylinders in $\widetilde{M}$. Recall that if $p\in\mathfrak{V}$ is a conic singularity of angle $n\pi$, then $\pi^{-1}(p)$ is formed by two conical singularities of angle $n\pi$ if $n$ is even, whereas if $n$ is odd $\pi^{-1}(p)$ is a conical singularity of angle $2n\pi$. Given that the multicurves $\alpha$ and $\beta$ are in minimal position, points in $\mathfrak{V}\cap\partial H_i$ which are conical singularities of angle $\pi$ are actually punctures of $S$.
This implies that $\widetilde{H}_{i_1}\cup\widetilde{H}_{i_2}$ cannot be merged within $\widetilde{M}$ into a flat cylinder. The same conclusion holds when $\mathfrak{V}\cap\partial H_i$ has conical singularities of angle different from $\pi$ and the claim follows. Analogously, we have that for every vertical  cylinder $V_j$ in $M$ the lift  $\pi^{-1}(V_j)$ is formed by two disjoint isometric copies $ \widetilde{V} _{j_1}$, $\widetilde{V} _{j_2}$  of $V_j$ and these are maximal vertical cylinders in $\widetilde{M}$.
The families $\widetilde{\mathcal{H}}=\{\widetilde{\mathcal{H}}_{i_1},\widetilde{\mathcal{H}}_{i_2}\}$ and $\widetilde{\mathcal{V}}=\{ \widetilde{\mathcal{V}}_{j_1}, \widetilde{\cV}_{j_2}\}$ define horizontal and vertical (maximal) cylinder decompositions of $\widetilde{M}$ respectively.

Let $\tilde{\alpha}$, $\tilde{\beta}$ denote the lifts to the orientation double cover of $\alpha$ and $\beta$ respectively. Given that the moduli of cylinders downstairs and upstairs is the same, we have a pair of affine multitwists $\widetilde{T_{\tilde{\alpha}}},\widetilde{T_{\tilde{\beta}}}\in\Aff(\widetilde{M})$ with $DT_\alpha=D\widetilde{T_{\tilde{\alpha}}}$ and $DT_\beta=D\widetilde{T_{\tilde{\beta}}}$ in $\PSL(2,\mathbb{R})$.
If we rewrite the word defining $f$ replacing each appearence of $T_\alpha$ with $\widetilde{T_{\tilde{\alpha}}}$ and each appearence of $T_\beta^{-1}$ with $\widetilde{T_{\tilde{\beta}}}^{-1}$ the result is an affine multitwist $\tilde{f}$ on $\widetilde{M}$ with $D \tilde{f}=Df$ in $\PSL(2,\mathbb{R})$. The eigendirections of $ \tilde{f}$
define a pair of transverse $\tilde{f}$-invariant measured foliations $\widetilde{\mathcal{F}}_u$ and $\widetilde{\mathcal{F}}_s$. Moreover, we have that $\widetilde{\mathcal{F}}_u=\pi^{-1}(\mathcal{F}_u)$ and $\widetilde{\mathcal{F}}_s=\pi^{-1}(\mathcal{F}_s)$ (\emph{i.e.} the projection $\pi$ sends leaves to leaves). Let
$\hat{\pi}:\widehat{\widetilde{M}}\to\widehat{M}$ be the continuous extension of the projection $\pi$ to the metric completions of $M$ and $\widetilde{M}$ and define $ \tilde{\mathfrak{V}} :=\hat{\pi}^{-1}(\mathfrak{V})$. Remark that $ \tilde{\mathfrak{V}} =\cup (\partial\widetilde{H_i}\cap\partial\widetilde{V_j})$, where the boundaries of the cylinders are taken in
$\widehat{\widetilde{M}}$. As with $M$, for every $q\in\widetilde{\mathfrak{V}}$ we define $\Sep_q(*)$ as the set of leaves of $*\in\{\widetilde{\mathcal{F}}^u,\widetilde{\mathcal{F}}^s\}$ based at $q$. In this context, the proof of Proposition~\ref{ref:DensitySeparatrices} for translation surfaces then applies to $\widetilde{M}$ and we get the following:


\begin{corollary}
	\label{Lemma:SeparatricesAreDense}
Let $q\in\widetilde{\mathfrak{V}}$. Then any separatrix in $\Sep_q(\widetilde{\mathcal{F}}^u)\cup\Sep_q(\widetilde{\mathcal{F}}^s)$ is dense in $\widetilde{M}$.
\end{corollary}
If separatrices are dense upstairs they are dense downstairs. This ends the proof of Proposition~\ref{ref:DensitySeparatrices}.


\end{proof}

Let now $p\in S$ be the marked puncture and $\Sep_p(\mathcal{F}^u)=\{\gamma_1,\ldots,\gamma_m\}$. We denote by $S_\mu$
a fixed complete hyperbolic structure on $S$ ($\mu$ stands for the metric) of the first kind and define the completed ray graph $R(S;p)$ with respect to $\mu$.
Remark that in $S_\mu$ the point $p$ becomes a cusp, \emph{i.e.} a point at infinity. In what follows we associate to each $\gamma_i$ a simple geodesic in $S_\mu$ based at $p$. For elements in $\Sep_p(\mathcal{F}^s)$ the arguments are analogous.
The ideas we present are largely inspired by the work of P. Levitt~\cite{Levitt83}.

Henceforth $\pi:\mathbb{D}\to S_\mu$ denotes the universal cover, $\Gamma<{\rm PSL(2,\mathbb{R})}$ the Fuchsian group for which $S_\mu=\mathbb{D}/\Gamma$, $\tilde{p}\in\partial\mathbb{D}$ a chosen point in lift of the cusp $p$ to $\partial\mathbb{D}$ and $\tilde{\gamma_i}=\tilde{\gamma_i}(\tilde{p})$ the unique lift of $\gamma_i$ to $\mathbb{D}$ based at $\tilde{p}$.

\emph{Claim:} $\tilde{\gamma_i}$ converges to two distincs points in $\partial\mathbb{D}$.

First remark that since $\gamma_i$ is not a loop, then it is sufficient to show that $\tilde{\gamma_i}(t)$ converges to a point when considering the parametrization that begins at $\tilde{p}$ and $t\to\infty$. Recall that in $S$, the point $p$ is in a region bounded by a $2m$-polygon whose sides belong to closed curves ${\alpha_1,\ldots,\alpha_m,\beta_1,\ldots,\beta_m}$ in $\alpha\cup\beta$. Each of these curves is transverse to the leaves of $\mathcal{F}^u$ and of $\mathcal{F}^s$. 
Up to making an isotopy, we can suppose without loss of generality that the first element in ${\alpha_1,\ldots,\alpha_m,\beta_1,\ldots,\beta_m}$ intersected by $\gamma_i(t)$ (for the natural parametrization used in the proof of Proposition~\ref{ref:DensitySeparatrices}) is $\alpha_j$. Given that $\alpha_j$ is transverse to $\mathcal{F}^u$ and $\gamma_i$ is dense in $M$ we have that $\gamma_i\cap\alpha_j$ is dense in $\alpha_j$. In consequence $\tilde{\gamma_i}$ intersects $\pi^{-1}(\alpha_j)$ infinitely often. Remark that $\tilde{\gamma_i}$ intersects a connected component of $\pi^{-1}(\alpha_j)$ at most once. Indeed, if this was not the case there would exists a disc $D$ embedded in $\mathbb{D}$ whose boundary $\partial D$ is formed by an arc in $\tilde{\gamma_i}$ and an arc contained in $\pi^{-1}(\alpha_j)$ transverse to $\widetilde{\mathcal{F}^u}:=\pi^{-1}(\mathcal{F}^u)$. This is impossible because all singularities of $\widetilde{\mathcal{F}^u}$ are saddles (in particular only a finite number of separatrices can steam from each one of them) and there is a finite number of them inside $D$. Then, all limit points in $\widetilde{\gamma_i}$ in $\mathbb{D}\cup\partial\mathbb{D}$ different from $\tilde{p}$ are in the intersection of an infinite family of nest domains $\mathbb{D}\cup\partial\mathbb{D}$ whose boundaries in $\mathbb{D}$
are components of $\pi^{-1}(\alpha_j)$. 
Moreover, this intersection is a single point $q_i\in\partial\mathbb{D}$ because it has to be connected and the endpoints in $\partial\mathbb{D}$ of components in $\pi^{-1}(\alpha_j)$ are dense in $\mathbb{D}$ since $\Gamma$ is Fuchsian of the first kind. This finishes the proof of our claim above.

We define thus $\widetilde{\delta_i}=\widetilde{\delta_i}(\tilde{p})$ to be the geodesic in $\mathbb{D}$ whose endpoints are $\tilde{p}$ and $q_i$ as above and  $\delta_i:=\pi(\widetilde{\delta_i})$. The geodesic $\delta_i$ is well defined: it does not depend on the lift of $\gamma_i$ based at $\tilde{p}$ we have chosen and if we changed $\tilde{p}$ by some $\tilde{p}'=g p$ for some $g\in\Gamma$ then by continuity $q_i'=g q_i$. On the other hand $\delta_i$ is simple: if this was not the case two components of $\pi^{-1}(\delta_i)$ would intersect and this implies that two components of $\pi^{-1}(\gamma_i)$ intersect, which is imposible since $\widetilde{\mathcal{F}^u}$ is a foliation. Remark that if $\gamma_i\neq \gamma_j$ then $\delta_i$ and $\delta_j$ are disjoint. If this was not the case then there would be a geodesic in $\pi^{-1}(\delta_i)$ intersecting a geodesic in $\pi^{-1}(\delta_j)$, but this would imply that a connected component of $\pi^{-1}(\gamma_i)$ intersects a connected component of $\pi^{-1}(\gamma_j)$, which is impossible since $\widetilde{\mathcal{F}}^u$ is a foliation.

Hence we can associate to the set of separatrices $\{\gamma_1,\ldots,\gamma_m\}$ a set of pairwise distinct simple geodesics $\{\delta_1,\ldots,\delta_m\}$ based at $p$. Remark that by construction this set is $f$-invariant. In what follows we show that $\{\delta_1,\ldots,\delta_m\}$ is a clique of high-filling rays. By applying the same arguments to the separatrices of $\mathcal{F}^u$ based at $p$ one obtains a different $f$-invariant clique of $m$ high-filling rays. These correspond to the only two points in the Gromov boundary of the loop graph $L(S;p)$ fixed by $f$.

 Let $\tilde{\delta}$  be a geodesic in $\mathbb{D}$ based at $\tilde{p}$ such that $\delta:=\pi(\tilde{\delta})$ is a simple geodesic in $S_\mu$ which does not belong to $\{\delta_1,\ldots,\delta_m\}$. We denote by $q$ the endpoint of $\tilde{\delta}$ which is different from $\tilde{p}$. Since every short ray or loop has a geodesic representative, it is sufficient to show that for every $i=1,\ldots,m$ there exists a (geodesic) component of $\pi^{-1}(\delta_i)$ which intersects $\tilde{\delta}$.  We recommend to use Figure \ref{fig:fol} as a guide for the next paragraph.

\begin{figure}[!ht]
\begin{center}
\includegraphics[scale=1]{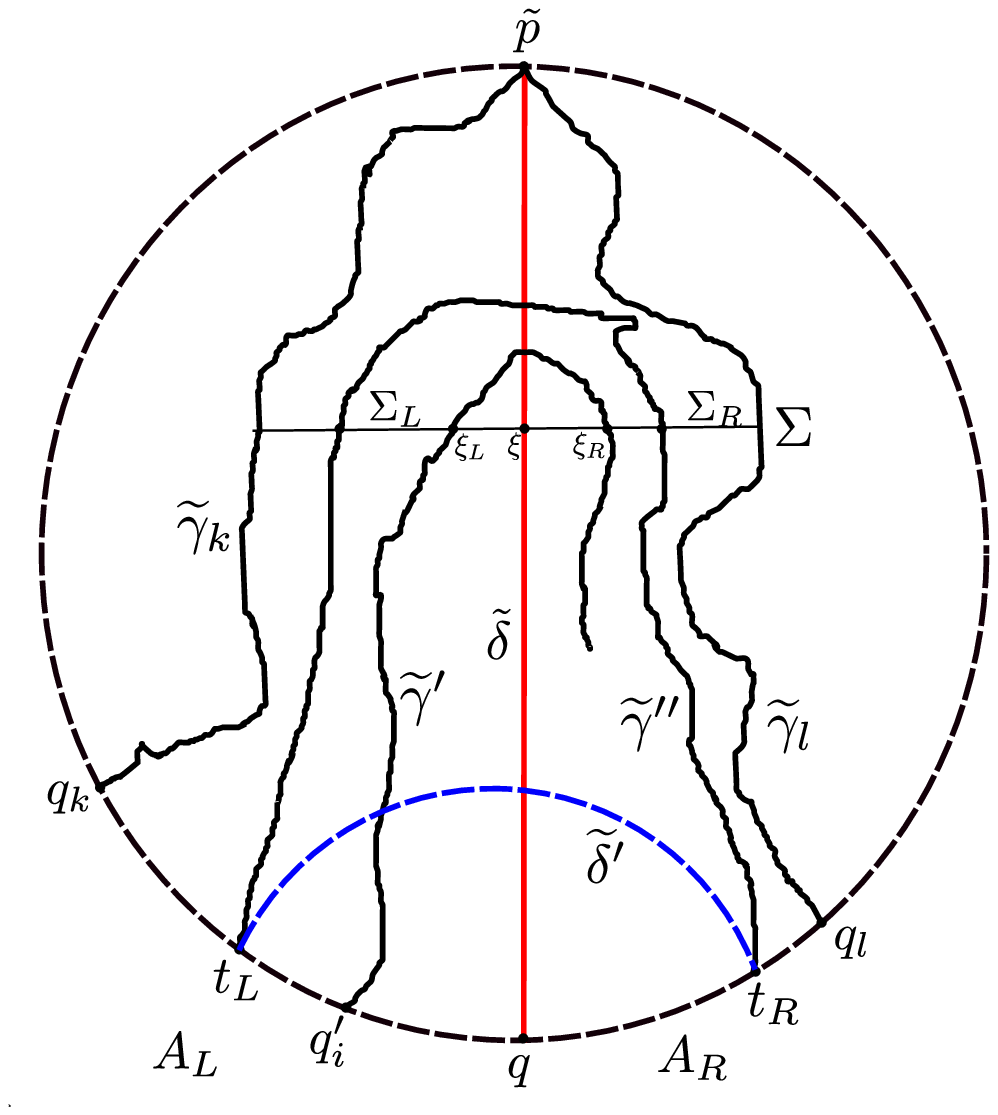}
\end{center}
\caption{}
\label{fig:fol}
\end{figure}



All components of $\pi^{-1}(\Sep_p(\mathcal{F}^u))$ with one endpoint in $\tilde{p}$ are of the form $\{g^k\tilde{\gamma}_1,\ldots g^k\tilde{\gamma}_m\}_{k\in\mathbb{Z}}$, with $g\in\Gamma$ parabolic fixing $\tilde{p}$.  Hence, there exists a closed disc $D\subset\mathbb{D}\cup\partial\mathbb{D}$
 whose boundary is formed by $\tilde{p}\cup \tilde{\gamma}_k\cup \tilde{\gamma}_l\cup A $, where $A$ is a closed arc in $\partial\mathbb{D}$ containing $q$; $\tilde{\gamma}_k$, $\tilde{\gamma}_l$ are (lifts of) separatrices and there is no element of  $\pi^{-1}(\Sep_p(\mathcal{F}^u))$ with one endpoint in $\tilde{p}$ in the interior of $D$. Remark that the endpoints of $A$ are
 $q_k$ and $q_l$ (the endpoints of $\tilde{\gamma}_k$ and $\tilde{\gamma}_l$ respectively). Given that $p$ is a saddle-type singularity of the foliation $\mathcal{F}^u$ there exists a neighbourhood of $\tilde{p}$ in $D$ which contains a segment\footnote{As a matter of fact this segment can be taked to live in one of the (lifts of) the curves in $\alpha\cup\beta$ forming the boundary of the disc in $S\setminus\alpha\cup\beta$ containing $p$.} $\Sigma$ with one endpoint in $\tilde{\gamma}_k$ and the other in $\tilde{\gamma}_l$, and which is transverse to $\widetilde{\mathcal{F}^u}$ except at one point $\xi$ in its interior. Moreover, since $p$ is an isolated singularity of $\mathcal{F}^u$, we can suppose that the closure of the connected component of $D\setminus\Sigma$ in $\mathbb{D}$ having $\tilde{p}$ in its boundary does not contain singular points of $\widetilde{\mathcal{F}^u}$ different from $\tilde{p}$. The point $\xi$ divides $\Sigma$ in two connected components $\Sigma_L$ and $\Sigma_R$. On the other hand $q$ divides $A$ in two connected components $A_L$ and $A_R$. Now let $i=1,\ldots,m$ be fixed. Since $\gamma_i$ is dense in $M$ we have that $\pi^{-1}(\gamma_i)$ is dense in $\mathbb{D}$ and in particular $\pi^{-1}(\gamma_i)\cap \Sigma$ is dense in $\Sigma$.
  Hence we can pick a leaf $\tilde{\gamma}_i'$ in $\pi^{-1}(\gamma_i)$ passing through a point $\xi_L\in\Sigma_L$ and suppose without loss of generality that one of its endpoints $q_i'$ is in $A_L$. Then $\tilde{\gamma}_i'\cap\Sigma=\{\xi_L,\xi_R\}$ with $\xi_R\in \Sigma_R$. Again, given that $\pi^{-1}(\gamma_i)\cap \Sigma$ is dense in $\Sigma$, we can find a leaf $\tilde{\gamma_i}''\in\pi^{-1}(\gamma_i)$ which intersects $\Sigma$ transversally
  at a point $\eta_R$ between $\xi_R$ and $\tilde{\gamma}_l$, and which has an enpoint $t_R$ in $A_R$ arbitrarly close to $q_l$. This is true because of the way $q_l$ was found: there is a family of connected components of $\pi^{-1}(\alpha)$, for some closed curve $\alpha$ in $M$ transverse to $\mathcal{F}^u$, bounding domains in $\mathbb{D}\cup\partial\mathbb{D}$ whose intersection is $q_l$. Now, by the way $\Sigma$ was chosen we have that $\tilde{\gamma}_i''\cap\Sigma=\{\eta_L,\eta_R\}$ with $\eta_L\in\Sigma_L$. This implies that $\tilde{\gamma}_i''$ has and
  endpoint $t_L$ in $A_L$. Hence the geodesic $\tilde{\delta}'$ determined by the endpoints $t_L$ and $t_R$ intersects $\tilde{\delta}$ and  $\delta_i=\pi(\tilde{\delta}')$ intersects $\delta=\pi(\tilde{\delta})$.\qed

\begin{remark}

In the proof of Theorem~\ref{THM:Loxodromics} we made use of the fact that every short-ray or loop has a geodesic representative, but this is not necessary. As a matter of fact the following is true: if $\tilde{\delta}$ is any curve in $\mathbb{D}$ based at $\widetilde{p}$ whose extremities define two different points in $\partial\mathbb{D}$ and such that $\pi(\tilde{\delta})=\delta$ is simple and does not belong to the set of separatrices $\{\gamma_1,\ldots,\gamma_m\}$, then for any $j=1,\ldots,m$ the geodesic $\delta_j$ intersects $\delta$.

On the other hand, in the proof of Theorem~\ref{THM:Loxodromics} the density of each separatrix of $\mathcal{F}^u$ or $\mathcal{F}^s$ on the \emph{whole} surface $S$ is not used. The proof remains valid if we only require separetrices of $\mathcal{F}^u$ and $\mathcal{F}^s$ to be dense on a subsurface $S'\subset S$ of finite type with enough topology, \emph{e.g.} such that all curves defining the polygon on which $p$ lives are essential in $S'$. In particular we have the following:
\end{remark}

\begin{corollary}
    Let $S'\subset S$ be an essential subsurface of finite topological type containing $p$ and $h\in\MCG(S')$ a pseudo-Anosov element for which $p$ is a $k$-prong for some $k\in\mathbb{N}$. Let $\hat{h}$ be the extension (as the identity) of $h$ to $S$.
    Then $\hat{h}$ is a loxodromic element of weight $k$. Moreover, the separatrices of the invariant transverse measured foliations of $h$ based at $p$ define\footnote{By a stretching process as described in the proof of Theorem~\ref{THM:Loxodromics}} the cliques of high-filling rays fixed by $\hat{h}$.
\end{corollary}

This result already appears in the work of Bavard and Walker~\cite{BaWa18B}, see Lemma 7.2.1 and Theorem 8.3.1.

\subsection{Proof of Theorem~\ref{THM:Multicurves}}

\subsubsection{Preliminaries}



In this section we present, for each infinite type surface, a model that is convenient for the proof of  Theorem~\ref{THM:Multicurves}.


\noindent{\bf{Normal forms for infinite-type surfaces}}. In what follows we detail how to construct, for any infinite type surface $S$, a graph $T(S)\subset\mathbb{H}^3$ having a regular neighbourhood whose boundary is homeomorphic to $S$ (our choice of ambient space obeys illustrative purposes only). There are many ways to construct such graph. The one we present is intended to make the proof of Theorem~\ref{THM:Multicurves} more transparent.

 Let $2^\N:=\prod_{j\in\N}\{0,1\}_j$ be the Cantor set. In general terms, the construction is as follows. We consider a rooted binary tree $T2^\N$, a homeomorphism $f:\prod_{j\in\N}\{0,1\}\to\Ends(T2^\N)$ from the standard binary Cantor set to the space of ends of this tree, and we choose a topological embedding $i:\Ends(S)\hookrightarrow \prod_{j\in\N}\{0,1\}$. We show that there exists a subtree  of $T2^\N$ whose space of ends is precisely $f\circ i(\Ends(S))$. For our purposes it is important that this subtree is \emph{simple} (see Definition~\ref{DEF:SimpleTree} below). Then, if $S$ has genus, we perform a surgery on vertices of the aforementioned subtree of of $T2^\N$ belonging to rays starting at the root and having one end on $f\circ i(\Ends_\infty(S))$.

\noindent \textbf{The rooted binary tree}. For every $n\in\N$ let $2^{(n)}:=\prod_{i=1}^n\{0,1\}$ and $\pi_i:2^{(n)}\to\{0,1\}$ the projection on to the $i^{th}$ coordinate. The rooted binary tree is the graph  $T2^\N$ whose vertex set $V(T2^\N)$ is the union of the simbol $\mathfrak{r}$ (this will be the root of the tree) with the set $\{D:D\in2^{(n)}\hspace{1mm}\text{for some $n\in\N$}\}$. The edges $E(T2^\N)$ are $\{(\mathfrak{r},0)$, $(\mathfrak{r},1)\}$ together with:
$$
\{(D,D'):D \in 2^{(n)}\hspace{1mm} D'\in2^{(n+1)}\hspace{1mm}\text{for some $n\in\N$},\hspace{1mm}\text{and}\hspace{1mm}\pi_i(D_s)=\pi_i(D_t)\hspace{1mm}\forall 1\leq i\leq n\}
$$
Henceforth $T2^\N$ is endowed with the combinatorial distance. For every $\hat{x}=(x_n)\in 2^\N$ we define $r(\hat{x})=(\mathfrak{r},a_1,\ldots,a_n):=(x_1,\ldots,x_n,\ldots)$ to be the infinite geodesic ray in $T2^N$ starting from $\mathfrak{r}$ and ending in $\Ends(T2^\N)$. Then, the map
\begin{equation}
f:\prod_{i\in\N}\{0,1\}\to\Ends(T2^\N)
\end{equation}
which associates to each infinite sequence $\hat{x}=(x_n)_{n\in\N}$ the end $f(\hat{x})$ of $T2^\N$ defined by the infinite geodesic ray $r(\hat{x})$ is a homeomorphism.

\begin{definition}
    \label{DEF:SimpleTree}
Let $v$ and $v^*$ two different vertices in a subtree $\mathcal{T}$ of $T2^\N$. If $v$ is contained in the geodesic which connects $v^*$ with $\mathfrak{r}$, then we say that $v^*$ is a \emph{descendant} of $v$. A connected rooted subtree of $\mathcal{T}$ without leaves is \emph{simple} if all descendants of a vertex $v\neq\mathfrak{r}$ of degree two, have also degree two.
\end{definition}

\begin{lemma}
    \label{Lemma:SimpleSubtrees}
Let $F\subset\Ends(T2^\N)$ be closed. Then there exists a \emph{simple} subtree $T$ of $T2^\N$ rooted at $\mathfrak{r}$ such that  $F$ is homeomorphic to $\Ends(T)$.
\end{lemma}

We postpone the proof of this lemma to the end of the section.

\begin{definition}
    \label{DEF:InducedSubtree}
Given a subset $F$ of $\Ends(T2^\N)$ we define $T_F:=\bigcup_{\hat{x}\in f^{-1}(F)} r(\hat{x})\subseteq T2^\N$ and call it the \emph{tree induced by $F$}.
\end{definition}

\textbf{Surgery}. Let $T$ be a subtree of $T2^\N$ rooted at $\mathfrak{r}$ and having no leaves different from this vertex, if any. Let $L$ be a subset of the vertex set of $T$. We denote by $\Gamma_{T,L}$ the graph obtained from $T$ and $L$ after performing the following operations on each vertex $v\in L$:
\begin{enumerate}
\item If $v$ has degree 3 with adjacent descendants $v', v''$  we delete first the edges $\{(v,v'),(v,v'',)\}$. Then we
 add to $L$ two vertices $v_*',v_*''$ and the edges $\{(v,v_*'),(v,v_*''),(v_*',v'),(v_*'',v''),(v'_*,v_*'')\}$.
\item If $v$ has degree 2 and $v'$ is its adjacent descendant, we delete first the edge $(v,v')$. Then we add to $L$ two vertices $v_*',v_*''$ and the edges $\{(v,v_*'),(v,v_*''),(v_*',v'),(v'_*,v_*'')\}$.
\end{enumerate}

\begin{definition}
    \label{DEF:Graph(S)}
Let $S$ be a surface of infinite type of genus $g\in\N\cup\{\infty\}$ and $\Ends_\infty(S)\subset\Ends(S)\subset 2^\N$ its space of ends accumulated by genus and space of ends, respectively. We define the graph $T(S)$ according to the following cases. In all of them we suppose w.l.o.g that $T_{f(\Ends(S))}$ is simple.
\begin{enumerate}
    \item If $g=0$ let $T(S):=T_{f(\Ends(S))}$,
    \item if $g\in\Z_{>0}$ let $T(S):=\Gamma_{T,L}$ where $T=T_{f(\Ends(S))}$ and $L={a_1,a_2,\ldots,a_g}$ and $(\mathfrak{r},a_1,\ldots,a_g,\ldots)=r(\hat{x})$ for some $\hat{x}\in\Ends(S)\subset 2^\N$, and
    \item if $g=\infty$ let $T(S):=\Gamma_{T,L}$ where $T=T_{f(\Ends(S))}$ and $L$ is the set of vertices of the subtree $T_{f(\Ends_{\infty}(S))}\subset T_{f(\Ends(S))}$.
\end{enumerate}
\end{definition}

By construction, there exists a geometric realization for $T(S)$ as a graph in the plane $\{(x,0,z)\in\mathbb{H}^3:z>0\}$ in 3-dimensional hyperbolic space, which we denote again by $T(S)$. Moreover there exists a closed regular neighbourhood $N(T(S))$ so that $S$ is homeomorphic to $S'=\partial N(T(S))$, see Figure~\ref{Embedding}. Observe that $T(S)$ is a strong deformation retract of $N(T(S))$.  We identify $S$ with $S'$, and we say that $S$ is in \emph{normal form}, and $T(S)$ is the \emph{underlying} graph that induces $S$.



\begin{figure}[!ht]
\begin{center}
	\includegraphics[width=.45\textwidth]{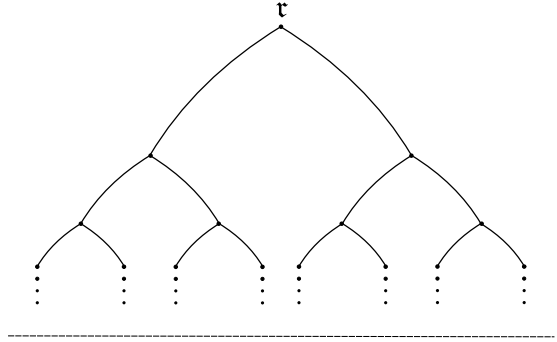}
	\caption{Embedding of the tree $T(S)$.}
	\label{Embedding}
\end{center}
\end{figure}

\begin{remark}
    \label{Remark:ContrasteconWalker&Juliette}
In \cite{BaWa18B}, A. Walker and J. Bavard carry out a similar construction. For this they introduce the notion of \emph{rooted core tree} $T$ from which they construct a surface $\Sigma(T)$ homeomorphic to a given infinite type surface $S$, see Lemma 2.3.1 in \cite{BaWa18B}. The graph $T_{f(\Ends(S))}$ (see Definition~\ref{DEF:Graph(S)}) is turned into a rooted core tree by declaring that the vertices of $T_{f(\Ends_\infty(S))}$ are the marked vertices. The main difference with the work of Walker and Bavard is that the normal form we are proposing comes from a \emph{simple} tree. This property is strongly used in the proof of Theorem~\ref{THM:Multicurves}.



\end{remark}





\emph{Proof of Lemma~\ref{Lemma:SimpleSubtrees}}. Let $T_F$ be the subtree of $T2^{\N}$ induced by $F$. Let $V'$ be the set of vertices of $T_F$ of degree 2, different from $\mathfrak{r}$, having at least one descendant of degree 3. Then $V'=\sqcup_{i\in I} V'_i$, where:
\begin{enumerate}
    \item $V_i'$ is a subset of the vertices of a ray $r(\hat{x})$ for some $\hat{x}\in f^{-1}(F)$,
    \item for every $i\in I$, one can label $V'_i=\{a_{i,1},\ldots,a_{i,k_i}\}$ so that $a_{i,l+1}$ is a descendant of $a_{i,l}$ adjacent to $a_{i,l}$.
    \item for every $i\in I$, the vertex $A_i$ of $T_F$ adjacent to $a_{i,1}$ other than $a_{i,2}$ is either the root $\mathfrak{r}$ or a vertex of degree 3. Similarly, the vertex $B_i$ of $T_F$ adjacent to $a_{i,k_i}$ other than $a_{i,k_i-1}$ is of degree 3.
\end{enumerate}
Replacing the finite simple path from $A_i$ to $B_i$ by an edge $(A_i,B_i)$ does not modify the space of ends of $T_F$. By doing this for every $i\in I$ we obtain a simple tree as desired. \qed

\emph{Proof of Theorem~\ref{THM:Multicurves}}. This proof is divided in two parts. First we show the existence of a pair of multicurves of finite type whose union fills $S$ and which satisfy (1) and (2). In the second part we use these to construct the desired multicurves $\alpha$ and $\beta$.

\medskip

\noindent {\bf{First part:}} Let $S$ be an infinite-type surface in its normal form, and $T(S)\subset\mathbb{H}^3$ the underlying graph that induces $S$. We are supposing that $T(S)$ is obtained after surgery from a simple tree as described above. The idea here is to construct two disjoint collections $A$ (blue curves) and $B$ (red curves) of pairwise disjoint curves in $S$ such that after forgetting the non-essential curves in $A\cup B$, we get the pair of multicurves  $\alpha$ and $\beta$  which satisfy (1) and (2) as in Theorem~\ref{THM:Multicurves}.


Let $T_g(S)$ be the full subgraph of $T(S)$ generated by all the vertices which define a triangle in $T(S)$. Observe that, since $T(S)$ is constructed by performing a surgery on a simple tree, the graph $T_g(S)$ is connected.

\medskip

Let $T_g'(S)$ be the subset obtained as the union of $T_g(S)$ with all the edges in $T(S)$ adjacent to $T_g(S)$. As $T_g(S)$ is connected, then $T_g'(S)$ is also connected. Let $\Delta$ be a triangle in $T_g(S)$, and $\Delta'$ be the disjoint union of $\Delta$ with all the edges in $T_g'(S)$ adjacent to $\Delta$. We notice that $\Delta'$ is one of the following two possibilities: (1) the disjoint union of $\Delta$ with exactly three edges adjacent to it, or (2) the disjoint union of $\Delta$ with exactly two edges adjacent to it. For each case, we choose blue a red curves in $S$ as indicated in the Figure \ref{ConstructingTheCurvesWithGenus}.

\begin{figure}[!ht]
\begin{center}
	\includegraphics[width=.75\textwidth]{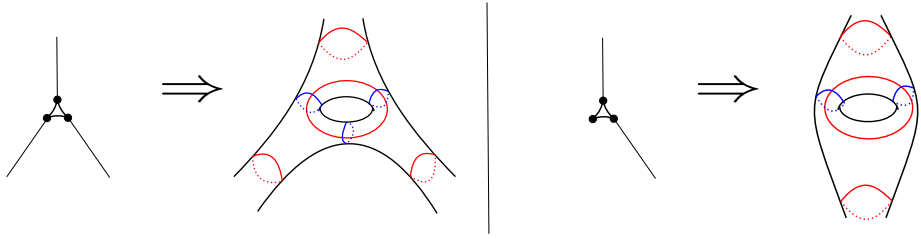}
	\caption{Blue and red curves associated to the neighborhood $\Delta'$ of a triangle $\Delta$.}
	\label{ConstructingTheCurvesWithGenus}
\end{center}
\end{figure}

For each edge $e$ in $T_g'(S)$ which connects two triangles in $T_g'(S)$, we choose a blue curve in $S$ as indicated in Figure \ref{ConectaGenero}.

\begin{figure}[!ht]
\begin{center}
	\includegraphics[width=.35\textwidth]{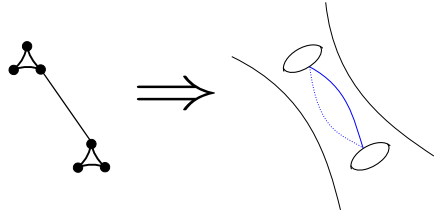}
	\caption{ Blue curve associated to an edge which connects two triangles.}
	\label{ConectaGenero}
\end{center}
\end{figure}
We consider the following cases.

\medskip
\underline{$\Ends(S)=\Ends_{\infty}(S)$}. In this case 
$A$ and $B$ are the multicurves formed by the blue and red curves as chosen above respectively.

\medskip

\underline{$\Ends(S)\neq\Ends_{\infty}(S)$}. Let $C$ be a connected component of $T(S)-T_g'(S)$. Given that $T(S)$ is obtained from a simple tree, $C$ is a tree with infinitely many vertices. Let $v$ be the only vertex in $C$ which is adjacent to an edge $e(v)$ in $T_g'(S)$. If $v$ has degree one in $C$, then every vertex of $C$ different from $v$ has degree two because $T_{f(\Ends(S))}$ is a simple subtree of $T2^\N$. In this case, we have that the subsurface $S(C) \subset S$ induced by $C$ is homeomorphic to a punctured disc. In particular, the red curve in $S$ associated to the edge $e(v)$ chosen as depicted in Figure~\ref{ConstructingTheCurvesWithGenus} is not essential in $S$.

Suppose now that $v$ has degree two in $C$. We color with blue all the edges in $C$ having vertices at combinatorial distances $k$ and $k+1$ from $v$ for every even $k\in\Z_{\geq 0}$. We color all other edges in $C$ in red, see the left-hand side in Figure~\ref{ColorEdgesOfF}. Let $e$ and $e'$ be two edges in $C$ of the same color and suppose that they shares a vertex $v$. Suppose that all vertices of $e \cup e'$ different from $v$ have degree three. If $e$ and $e'$ are marked with blue color (respectively red color), we choose the red curve (respect. blue curve) in $S$ as in the right-hand side of Figure~\ref{ColorEdgesOfF}.

\begin{figure}[!ht]
\begin{center}
	\includegraphics[width=.75\textwidth]{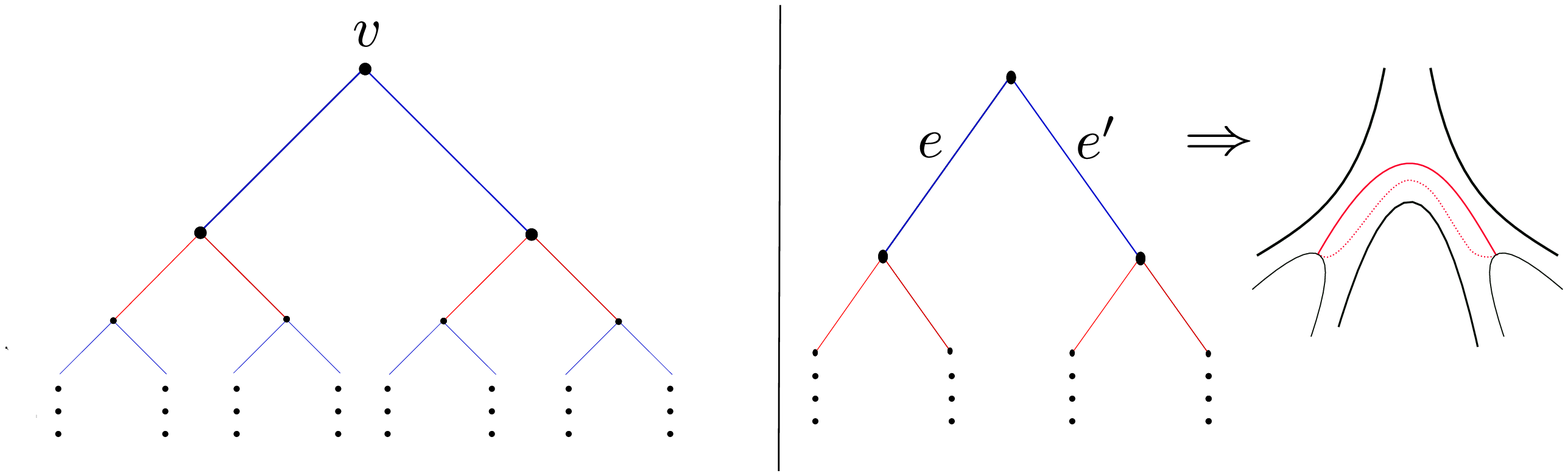}
	\caption{(Left) Edges of $C$ colored with blue and red alternating in levels. (Right) The corresponding curve in $S$ for the pair of edges $e$ and $e^\prime$ marked with blue color. }
	\label{ColorEdgesOfF}
\end{center}
\end{figure}

For the edge $e(v)\in T'_g(S)$, we choose the blue curve in $S$ as in the left-hand side of Figure~\ref{ReplaceBridgeEdge}. Finally, for each edge $e$ of $C$, we take a curve in $S$ with the same marked color of $e$ as is showed in the right-hand side of Figure~\ref{ReplaceBridgeEdge}.

\begin{figure}[!ht]
\begin{center}
	\includegraphics[width=.75\textwidth]{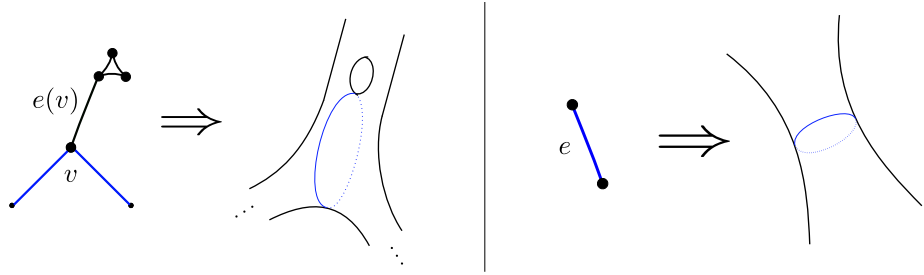}
	\caption{(Left) The corresponding curve for the $e$ which connects $F$ with $T_g'(S)$. (Right) } Blue (red) curve associated to an edge $e$ of $F$.
	\label{ReplaceBridgeEdge}
\end{center}
\end{figure}

If $\Ends(S)$ has at most one isolated planar end here ends the construction of the multicurves $A$ and $B$.

\medskip

Now suppose that $\Ends(S)$ has more that one isolated planar end, that is, $S$ has at least two punctures. Let $R$ be the full subgraph of $T(S)$ generated by all the vertices of degree 3 in $T(S)$ together with the root vertex $\mathfrak{r}$, and define $T_{g}''(S)$ as the full subgraph of $T(S)$ generated by all the vertices in $T(S)$ at distance at most 1 from $R$. The graph $T_{g}''(S)$ is connected (again, because $T_{f(\Ends(S))}$ is simple) and contains $T_{g}'(S)$ . It also has at least two leaves, i.e., vertices of degree one which we denote by $v_1$ and $v_2$. If $v_1$ and $v_2$ are at distance 3 in $T_{g}''(S)$ there exist a single edge $e$ in $T_{g}''(S)$ whose adjacent vertices are at distance 1 from $v_1$ or $v_2$. Let us suppose that this is not an edge of a triangle in $T(S)$. Then $e$ is contained in a connected component of $T(S)\smallsetminus T_{g}'(S)$. In this case, if $e$ is marked with red color (blue color), we choose the blue curve (red curve) in $S$ as is shown in Figure \ref{ReplaceEdgesConnectsLeaves}. In all other cases we do nothing and this finishes the construction of the multicurves $A$ and $B$.

\begin{figure}[!ht]
\begin{center}
	\includegraphics[width=.50\textwidth]{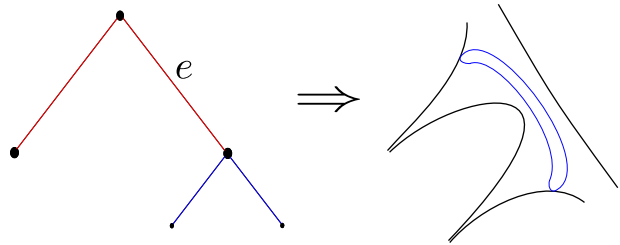}
	\caption{The corresponding blue curve in $S$ for the red edge $e$.}
	\label{ReplaceEdgesConnectsLeaves}
\end{center}
\end{figure}

We define $\alpha=\{\alpha_i\}_{i\in I}$ and $\beta=\{\beta_j\}_{j\in J}$ as the set of essential curves in $A$ and $B$, respectively. By construction, $\alpha$ and $\beta$ are multicurves of finite type in minimal position and $\rm{i}(\alpha_i,\beta_j)\leq 2$ for every $i\in I$ and $j\in J$. Upon investigation, one can observe that each connected component of $S\setminus\alpha\cup\beta$ is a disc or a punctured disc whose boundary is formed by 2, 4, 6 or 8 segments.

\medskip

\noindent {\bf{Second part:}} Let $m\in \mathbb{N}$. Take a finite multicurve $\delta$ in $S$ such that, if $Q$ is the connected component of $S \smallsetminus \delta$ which contains $p$, we have that $Q\setminus p$ is homeomorphic to $S_0^{m+3}$, i.e., a genus zero surfaces with $m+3$ punctures. In $Q$ we choose blue and red curves to form a chain as in Figure \ref{MulticurvesInPointP} and color them in blue and red so that no two curves of the same color intersect. We denote the blue and red curves in $Q$ by $\alpha'$ and $\beta'$ respectively. Remark that  the connected component of $Q \smallsetminus (\alpha'\cup \beta')$ containing the point $p$ is a $2m-$polygon.


\begin{figure}[!ht]
\begin{center}
	\includegraphics[scale=.5]{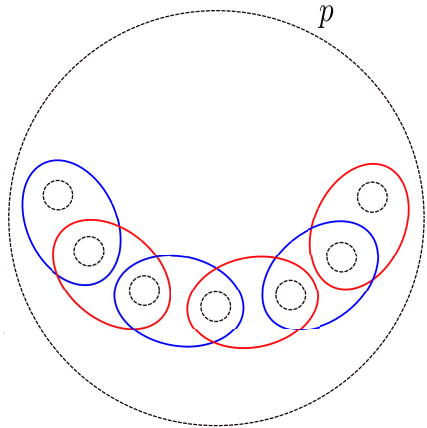}
	\caption{The multicurves $\alpha'$ (blue) and $\beta'$ (red).}
	\label{MulticurvesInPointP}
\end{center}
\end{figure}

The idea is to extend $\alpha^\prime$ and $\beta^\prime$ to multicurves $\alpha$ and $\beta$, respectively, which satisfy all the desired properties. We consider two cases for the rest of the proof.

\underline{$m$ even}. Without loss of generality we suppose that all punctures in $Q$ different from $p$ are encircled by elements of $\alpha'$. Now, let $F$ be a connected component of $S\smallsetminus \alpha^\prime$ not containing the point $p$. Then the closure $\overline{F}$ of $F$ in $S$ is a surface with $b>0$ boundary components. Moreover, $\overline{F}\cap \beta^\prime$ is a finite collection of disjoint essential arcs in $\overline{F}$ with end points in $\partial \overline{F}$, and the end points of an arc in $\overline{F}\cap \beta^\prime$ are in a common connected component of $\partial \overline{F}$. We denote by $\theta_F$ the collection of arcs in $\overline{F}\cap \beta^\prime$, and by $\delta_F$ to be the set of curves in $\delta$ contained in $F$.

\noindent {\bf{Claim:}} There exists a pair of multicurves $\alpha_F''$ and $\beta_F''$ whose union fills $F$, which satisfy (1) \& (2) in Theorem~\ref{THM:Multicurves} and such that $\theta_F \cap \beta_F''=\emptyset$.

Remark that if we define $\alpha:=\alpha^\prime \bigcup \left(\bigcup_{F\subset S\setminus\alpha'} \alpha_F'' \right)$ and $\beta:=\beta^\prime \bigcup \left(\bigcup_{F\subset S\setminus\alpha'} \beta_F'' \right)$, then $\alpha$ and $\beta$ are the desired pair of multicurves.

We divided the proof of our claim in two cases: $b=1$ and $b>1$.

\noindent {\bf {Case $b=1$}}. If $F$ is a finite-type surface is it not difficult to find the multicurves $\alpha_F''$ and $\beta_F''$. If $F$ is of infinite-type let $\alpha''_F$ and $\beta''_F$ be the blue and red curves obtained from applying the first part of the proof of Theorem~\ref{THM:Multicurves} to $F$. Remark that by construction all arcs in $\theta_F$ intersect only one curve in $\alpha_F''\cup\beta_F''$, hence, up to interchanging the colors of $\alpha''_F$ and $\beta_F''$, we can get that $\beta_F''\cap\theta_F=\emptyset$.

\noindent {\bf{Case $b>1$}}. Again, the case when $F$ is a finite-type surface is left to the reader. If $F$ is of infinite type,  let $\gamma$ be the separating curve in $\overline{F}$ which bounds a subsurface $W\subset\overline{F}$ of genus 0 with one puncture, $b$ boundary components and such that $\partial W=\partial\overline{F}$ and write $\overline{F}\smallsetminus \gamma=W\sqcup F_1$. Let $\theta_{F_1}$ to be the set of arcs given by $\theta_F\cap \overline{F_1}$. Let $\eta_1$ and $\eta_2$ be two curves in $F_1$ (not necessary essential) such that $\gamma,\eta_1$ and $\eta_2$ bounds a pair of pants $P$ in $F_1$. If an element of $\theta_F$ intersects $\eta_1 \cup \eta_2$ then we replace it with one that doesn't and which is disjoint from all other arcs in $\theta_F$. Up to making these replacements, we can assume that $\theta_F$ does not intersects $\eta_1 \cup \eta_2$, see Figure~\ref{CasobMayor1A}. Hence, $\theta_{F_1}\subseteq \overline{P}$. As $\overline{F_1}$ has one boundary component, by the case $b=1$ above, there exist a pair of multicurves $\alpha_{F_1}''$ and $\beta_{F_1}''$ which fills $F_1$ and such that $\theta_{F_1}\cap \beta_{F_1}''=\emptyset$. Define then $\alpha_F^{''}:=\{\gamma \}\cup \alpha_{F_{1}}^{''}$ and $\beta_F^{''}:=\beta_{F_{2}}^{''}$.

\begin{figure}[!ht]
\begin{center}
	\includegraphics[scale=.5]{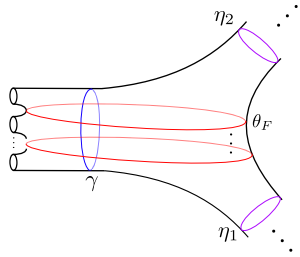}
	\caption{Case $b>1$.}
	\label{CasobMayor1A}
\end{center}
\end{figure}

\underline{m odd}. Without loss of generality we suppose that $\alpha'$ encircles all punctures in $Q$ different from $p$ except one. We add the curves $\alpha_1$ and $\beta_1$ to $\alpha'$ and $\beta'$  as depicted in Figure~\ref{CasobMayor1B} respectively. Then we consider each connected component $F$ of $S\setminus\alpha'$ and proceed as in the preceding case.
\qed

\begin{figure}[!ht]
\begin{center}
	\includegraphics[width=.35\textwidth]{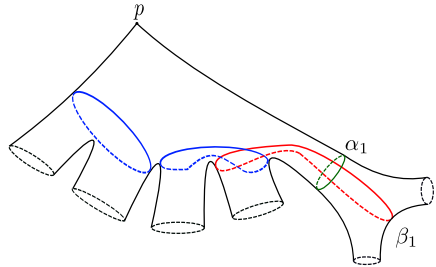}
	\caption{Case $m$ odd.}
	\label{CasobMayor1B}
\end{center}
\end{figure}

\begin{remark}
If $\alpha$ and $\beta$ are multicurves as constructed in the proof of Theorem~\ref{THM:Multicurves}, then $S\setminus\alpha\cup\beta$ is a family of polygons, each of which has either 2, 4, 6 or 8 sides. Hence, if $M=M(\alpha,\beta,\textbf{h})$ is given by the Hooper-Thurston-Veech construction, then the set $\mathfrak{V}$ defined in the proof of Theorem~\ref{THM:Loxodromics} is formed by regular points and conic singularities of total angle $\pi$, $3\pi$ or $4\pi$.
\end{remark}

\subsection{Proof of Corollary~\ref{Corollary:LoxodromicsConverging2Elliptic}} Our arguments use Figure~\ref{Fig:LoxosConvergeEliptico}. Let $\alpha$ and $\beta$ be the multicurves in blue and red illustrated in the figure; the union of these fills the surface $S$ in question (a Loch Ness monster). Let us write $\beta=\beta'\sqcup\{a_i\}_{i\in\N}\sqcup\{b_i\}_{i\in\N}$, and for each $n\in\N$ let $\beta_n:=\beta'\sqcup\{a_i\}_{i\geq n}\sqcup\{b_i\}_{i\geq n}$. Theorem~\ref{THM:Loxodromics} implies that $f_n:=T_\alpha\circ T_{\beta_n}^{-1}$ acts loxodromically on the loop graph $L(S;p)$ and hence it acts loxodromically on the main component of the completed ray graph $\mathcal{R}(S;p)$ (see Theorem~\ref{THM:CompletedRayGraph}). Remark that $f_n$ converges to $f=T_\alpha\circ T_{\beta'}^{-1}$ in the compact-open topology. On the other hand, $f$ fixes the short rays $l$ and $l'$ and hence it acts elliptically on both $\mathcal{R}(S;p)$ and $L(S;p)$. \qed

\begin{remark}
Using techniques similar to the ones presented in the proof of Theorem~\ref{THM:Multicurves} one can construct explicit sequences $(f_n)$ as above for any infinite-type surface $S$.
\end{remark}

\begin{figure}[!ht]
\begin{center}
	\includegraphics[scale=.5]{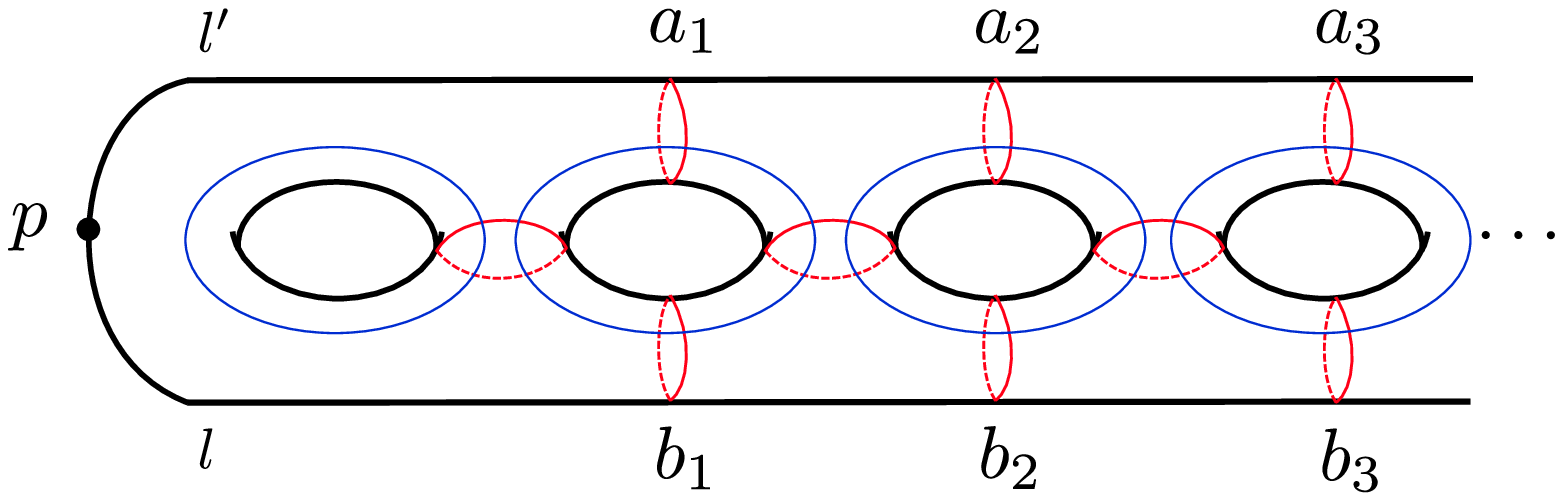}
	\caption{}
	\label{Fig:LoxosConvergeEliptico}
\end{center}
\end{figure}

\end{document}